\documentclass[11pt,reqno]{amsart}
\usepackage{fullpage}
\usepackage{mathtools}
\usepackage{latexsym}
\usepackage{amsmath,amsfonts,amsthm,amssymb,amscd}
\usepackage{hyperref}
\usepackage{url}
\usepackage[usenames,dvipsnames]{color}
\usepackage{enumitem}
\usepackage{fourier}
\usepackage{calrsfs}
\DeclareMathAlphabet{\pazocal}{OMS}{zplm}{m}{n}
\usepackage[T1]{fontenc}
\usepackage[utf8]{inputenc}
\usepackage{verbatim}

\newtheorem{theorem}{Theorem}[section]
\newtheorem{thmx}{Hypothesis}
\newtheorem{corollary}[theorem]{Corollary}
\newtheorem{lemma}[theorem]{Lemma}
\newtheorem{prop}[theorem]{Proposition}

\numberwithin{equation}{section}

\renewcommand{\leq}{\leqslant}
\renewcommand{\geq}{\geqslant}

\newcommand{\R}{\ensuremath{\mathbf{R}}}
\newcommand{\C}{\ensuremath{\mathbf{C}}}
\newcommand{\Q}{\ensuremath{\mathbf{Q}}}

\newcommand{\ep}{\epsilon}
\newcommand{\vep}{\varepsilon}

\newcommand{\ra}{\rightarrow}

\newcommand{\Hyp}[1]{\textbf{P#1}}
\newcommand\be{\begin{equation}}
\newcommand\ee{\end{equation}}
\newcommand\bi{\begin{itemize}}
\newcommand\ei{\end{itemize}}
\newcommand\ben{\begin{enumerate}}
\newcommand\een{\end{enumerate}}
\renewcommand{\d}[1]{\;\operatorname*{d}\!#1}

\newcommand\gammaf{\gamma_{\!\!f}}
\newcommand\tgammaf{\widetilde{\gamma}_{\!\!f}}
\newcommand\gammaL{\gamma_{\!L}}
\newcommand\Phii{\Phi_f}
\newcommand\Phiidual{\Phi_{\overline f}}
\newcommand{\selberg}{\pazocal S}

\newlist{Hlist}{enumerate}{1}
\setlist[Hlist]{label=\Hyp{\arabic*}}
\newlist{selberglist}{enumerate}{1}
\setlist[selberglist]{label=\roman*}

\DeclareMathOperator{\res}{res}

\DeclareMathOperator*{\GL}{GL}
\DeclareMathOperator{\erfc}{erfc}

\title[Gaps between zeros of $\GL(2)$ $L$-functions]{Gaps between zeros of $\GL(2)$ $L$-functions}

\author[O. Barrett]{Owen Barrett}
\address{Department of Mathematics, Yale University, New Haven, CT 06511, USA}
\email{Owen.Barrett@gmail.com}

\author[B. McDonald]{Brian McDonald}
\address{Department of Mathematics, University of Rochester, Rochester, NY 14627, USA}
\email{bmcdon11@u.rochester.edu}

\author[S.\,J. Miller]{Steven J. Miller}
\address{Department of Mathematics \& Statistics, Williams College, Williamstown, MA 01267, USA}
\email{sjm1@williams.edu, Steven.Miller.MC.96@aya.yale.edu}

\author[P. Ryan]{Patrick Ryan}
\address{Department of Mathematics, Harvard University, Cambridge, MA 02138, USA}
\email{ryan880@gmail.com}

\author[C. Turnage-Butterbaugh]{Caroline L. Turnage-Butterbaugh}
\address{Department of Mathematics, North Dakota State University, Fargo, ND 58108, USA}
\email{cturnagebutterbaugh@gmail.com}

\author[K. Winsor]{Karl Winsor}
\address{Department of Mathematics, University of Michigan\\Ann Arbor, MI 48109}
\email{krlwnsr@umich.edu}

\begin{document}


\begin{abstract}
Let $L(s,f)$ be an $L$-function associated to a primitive (holomorphic or Maass) cusp form $f$ on $\GL(2)$ over $\Q$. Combining mean-value estimates of Montgomery and Vaughan with a method of Ramachandra, we prove a formula for the mixed second moments of derivatives of $L(1/2+it,f)$ and, via a method of Hall, use it to show that there are infinitely many gaps between consecutive zeros of $L(s,f)$ along the critical line that are at least $\sqrt 3 = 1.732\ldots$ times the average spacing. Using general pair correlation results due to Murty and Perelli in conjunction with a technique of Montgomery, we also prove the existence of small gaps between zeros of any primitive $L$-function of the Selberg class. In particular, when $f$ is a primitive holomorphic cusp form on $\GL(2)$ over $\Q$, we prove that there are infinitely many gaps between consecutive zeros of $L(s,f)$ along the critical line that are at most $0.823$ times the average spacing. 
\end{abstract}

\maketitle

\tableofcontents

\vspace{.25in}


\section{Introduction}
Let $f$ be a primitive form on $\GL(2)$ over $\Q$ with level $q$, which we consider to be fixed. Then $f$
corresponds either to a primitive holomorphic cusp form or to a primitive Maass cusp
form. For $\Re(s)\gg1$, let
\begin{equation}\label{Ldefintion}
L(s,f)\ := \ \sum_{n=1}^{\infty}\frac{a_f(n)}{n^s}
\end{equation}
denote the $L$-function of degree 2 associated to $f$ as defined by Godement and
Jacquet~\cite{GJ}. Here, for any choice of $f$, we are normalizing so that
$a_f(1)=1$ and the critical line is $\Re(s)=1/2$. We study the
vertical distribution of the nontrivial zeros of $L(s,f)$, which we denote by
$\rho_{\!f} = \beta_{\!f} + i\gammaf$, where $\beta_{\!f}, \gammaf\in\R$ and
$0<\beta_{\!f}<1$. The analogous problems for normalized gaps between consecutive zeros of the Riemann zeta-function and for the Dedekind zeta-function of a quadratic number field have been studied extensively. For example, see {\cite{Bredberg,Bui,Bui2,Bui3,BHTB,BMN,CGGGHB,CGG,CGG2,FW,GGOS,Hall,Hall1, Montgomery, MO, Mu, ng1, Sound, caroline}. 

It is known (see Theorem 5.8 of \cite{IK}) that
\begin{equation}\label{countzeros}
  N(T,f)\ := \ \sum_{0<\gammaf \leq T} 1
  \ = \ \frac{T}{\pi}\log \frac{\sqrt{|q|}T}{2\pi e} +O\left(\log \mathfrak q(iT,f)\right)
\end{equation}
for $T\geq1$ with an implied absolute constant. Here $\mathfrak q(s,f)$ denotes the analytic conductor of $L(s,f)$. Consider the
sequence $0 \leq \gammaf(1)\leq \gammaf(2) \leq \cdots\leq \gammaf(n) \leq \cdots$
of consecutive ordinates of the nontrivial zeros of $L(s,f)$. By \eqref{countzeros}, it follows that the average size of
$\gammaf(n+1) - \gammaf(n)$ is
\begin{equation}
\frac{\pi}{\log\left(\sqrt{|q|}\gammaf(n)\right)}.
\end{equation}
Let
\begin{equation}\label{eq:mulambdadefs}
  \Lambda_f\ := \  \limsup_{n\to \infty}\,\frac{\tgammaf(n+1)-\tgammaf(n)}{\pi / \log\left(\sqrt{|q|}\tgammaf(n)\right)}
  \qquad \text{and} \qquad
  \lambda_f\ := \  \limsup_{n\to \infty}\,
  \frac{\gammaf(n+1)-\gammaf(n)}{\pi / \log\left(\sqrt{|q|}\gammaf(n)\right)},
\end{equation}
where $\tgammaf(n)$ corresponds to the $n$th nontrivial zero of $L(s,f)$ on the critical line $\Re(s) = 1/2$. Note that under the assumption of the Generalized Riemann Hypothesis for $L(s,f)$, we have $\Lambda_f = \lambda_f$. Unconditionally, it is certainly true that $\Lambda_f  \geq\lambda_f\geq1$, however we expect that $\Lambda_f = \lambda_f= \infty$. Towards a lower bound on $\lambda_f$, we prove the following unconditional result for $\Lambda_f$.

\begin{theorem}\label{thm:gaps}Let $L(s,f)$ be a primitive $L$-function on $\GL(2)$ over $\Q$. Then $\Lambda_f \geq\sqrt{3}=1.732\ldots$.
\end{theorem}

\begin{corollary}\label{thm:RHlargegaps}Let $L(s,f)$ be a primitive $L$-function on $\GL(2)$ over $\Q$. Then, assuming the Generalized Riemann Hypothesis for $L(s,f)$, we have $\Lambda_f=\lambda_f \geq\sqrt{3}=1.732\ldots$.
\end{corollary}

We also consider the question of small gaps between nontrivial zeros of $L(s,f)$, however our arguments may be applied to any primitive $L$-function, $L(s)$, in the Selberg class,\footnote{We define the Selberg class in Section \ref{properties}.} which we denote by $\selberg$. It is conjectured that all `standard' automorphic $L$-functions as described by Langlands are members of the Selberg class, but this is far from established. It is known, however, that  the primitive automorphic cuspidal $\GL(2)$ $L$-functions we consider are members of $\selberg$, and moreover membership in $\selberg$ has been established for primitive holomorphic cusp forms. This has not yet been established, however, in the real-analytic case of Maass cusp forms. 

Fix any $L\in\selberg$. As above, consider the sequence $0 \le\gammaL(1)\leq \gammaL(2) \leq \cdots\leq \gammaL(n) \leq \cdots$
of consecutive ordinates of the nontrivial zeros of $L(s)$, and define
\begin{equation}\label{eq:numLzeros}
N(T,L)\ := \ \sum_{0<\gammaL \leq T} 1 
\end{equation}
and
\be\label{eq:muLdef}
\mu_L\ := \  \liminf_{n\to \infty}\,\frac{\gammaL(n\!+\!1)-\gammaL(n)}{\pi/\log \gammaL(n)}.
\ee

By definition, we have $\mu_L\le1$, but we expect that $\mu_L=0$. It is conjectured that all the nontrivial zeros of $L(1/2+it)\in \selberg$ are simple, except for a possible multiple zero at the central point $s=1/2$. In the case that $f$ is a primitive holomorphic cusp form, Milinovich and Ng~\cite{micah-nathan} have shown, under the Generalized Riemann Hypothesis for $L(s,f)$, that the number of simple zeros of $L(s,f)$ satisfying $0<\gammaf\leq T$ is greater than a positive constant times $T(\log T)^{-\varepsilon}$ for any $\varepsilon>0$ and $T$ sufficiently large.

Every $L \in\selberg$ satisfies
  \be\log L(s)\ = \ \sum_{n=1}^\infty \frac{b_L(n)}{n^{s}},\ee
  where $b_L(n)=0$ unless $n=p^\ell$ for some $\ell\geq 1$, and $b_L(n)\ll n^\theta$
  for some $\theta<1/2$. In addition to the Generalized Riemann Hypothesis, we make the following assumption in the proof of small gaps between nontrivial zeros of $L\in\selberg$.
\begin{thmx}\label{eq:selbergconj}
Let $\upsilon_L(n)\ := \ b_L(n)\log n.$ We have
\be
\sum_{n\leq x}\upsilon_L(n)\overline{\upsilon_L(n)}\ = \ (1+o(1))x\log x\ee
as $x\to \infty$.
\end{thmx}
Hypothesis~\ref{eq:selbergconj}, proposed by Murty and Perelli in \cite{murty-perelli}, is a mild assumption concerning the correlation of the coefficients of $L$-functions at primes and prime powers. Hypothesis~\ref{eq:selbergconj} is motivated in \cite{murty-perelli} by the Selberg Orthogonality Conjectures, which are known to hold for $L$-functions attached to irreducible cuspidal automorphic representations on $\GL(m)$ over $\Q$ if $m\leq 4$. (If $m\le4$, see \cite{AS,LWY,LY1,LY2}.) 

In the case that $m_L=1$, it has recently been shown in \cite{CCLM} that $\mu_L \leq 0.606894$. We generalize these arguments for any $L\in \selberg$ to prove the following upper bounds on $\mu_L$. 

\begin{theorem}\label{thm:smallgaps}Let $L\in \selberg$ be primitive of degree $m_L$. Assume the Generalized Riemann Hypothesis and Hypothesis~\ref{eq:selbergconj}. Then there is a computable nontrivial upper bound on $\mu_L$ depending on $m_L$. In particular, we record the following upper bounds for $\mu_L$:
  \be\begin{tabular}{c c c}
    $m_L$ & & upper bound for $\mu_L$ \\
        \hline
    1 && 0.606894 \\
    2 && 0.822897 \\
    3 && 0.905604 \\
    4 && 0.942914 \\
    5 && 0.962190\\
    \vdots && \vdots 
  \end{tabular}\ee
where the nontrivial upper bounds for $\mu_L$ approach 1 as $m_L$ increases.
\end{theorem}

\begin{corollary}
Let $f$ be a primitive holomorphic cusp form on $\GL(2)$ over $\Q$ with level $q$ and $L(s,f)$ the associated $L$-function. Then $L(s,f) \in \selberg$, so assuming the Generalized Riemann Hypothesis and Hypothesis~\ref{eq:selbergconj}, we have
\begin{equation}
\liminf_{n\to \infty}
  \frac{\gammaf(n\!+\!1)-\gammaf(n)}{\pi / \log\left(\sqrt{|q|}\gammaf(n)\right)} \ <\  0.822897.
\end{equation}
\end{corollary}

The majority of this article is focused on deriving the lower bound given in Theorem \ref{thm:gaps}, which is accomplished using a method of
Hall~\cite{Hall} and some ideas of Bredberg~\cite{Bredberg}. We closely follow
the arguments in~\cite{caroline}, where the problem is considered for large gaps
between consecutive zeros of a Dedekind zeta-function of a quadratic number
field. Note that the $L$-function considered in ~\cite{caroline} is of degree 2, however it
is not primitive because it factors as the product of the Riemann zeta-function
and a Dirichlet $L$-function. Due to the primitivity of the $L$-functions in
the present work, we must consider a Rankin-Selberg type convolution, which we define by
\begin{equation}\label{convolution}
L(s,f \!\times\! \overline{f}) \ := \  \sum_{n=1}^{\infty}\frac{|a_f(n)|^2}{n^s},
\end{equation}
for $\Re(s)>1$. It can be shown that this function extends meromorphically to $\C$ and has a simple pole at $s=1$. For any choice of $f$, we let $c_f$ denote the residue of the simple pole of $L(s,f \!\times\! \overline{f})$ at $s=1$.

Following \cite{caroline}, we require asymptotic estimates of the mixed second
moments of $L(1/2+\!it,f)$ and $L^{\prime}(1/2+\!it,f)$ with a uniform error.
We obtain these by way of the following theorem.

\begin{theorem}\label{thm:mixed} Let $L(s,f)$ be a primitive $L$-function on $\GL(2)$ over $\Q$. Let $s=1/2+it$, $T$ large, and let $\mu, \nu$ denote non-negative integers. We have
\begin{align}
  \int_T^{2T}L^{(\mu)}\left(\tfrac{1}{2}\!+\!it,f\right)
  L^{(\nu)}\left(\tfrac{1}{2}\!-\!it,\overline{f}\right)\d t
  \ = \  \frac{(-1)^{\mu+\nu}2^{\mu+\nu+1}}{\mu\!+\!\nu\!+\!1}&c_fT\left(\log T\right)^{\mu+\nu+1}
  \ + \
  O\left(\mu!\nu!T(\log T)^{\mu+\nu}\right)
\end{align}
as $T\to \infty$, where $c_f$ denotes the residue of the simple pole of $L(s,f,\times \overline{f})$ at $s=1$, and the error term is uniform in $\mu$ and $\nu$.
\end{theorem}

If $f$ is a cusp form of even weight (at least 12) with respect to the full modular group, the cases $\mu= \nu$ are known. The case $\mu= \nu =0$ in this setting was proved by Good~\cite{Good}, and for any nonnegative integer $m$, the general case
$\mu = \nu =m$ was recently given by Yashiro~\cite{yashiro}. In addition
to the cases $\mu=\nu=0$ and $\mu=\nu=1$, we require the mixed case
$\mu=1, \nu=0$ in the proof of Theorem \ref{thm:gaps}. In this article, we first prove a shifted moment result and then obtain the more general formula given in Theorem \ref{thm:mixed} via differentiation (with respect to the shifts) and Cauchy's integral formula. We deduce the required shifted moment result, given below and proved in Section~\ref{sec:shiftproof} , using a method of Ramamchandra~\cite{ramachandra}. 

\begin{theorem}\label{thm:shiftedmoment}Let $L(s,f)$ be a primitive $L$-function on $\GL(2)$ over $\Q$. Let $s=1/2+it$, $T$ large, and $\alpha, \beta\in\C$ such that
  $|\alpha|,|\beta| \ll 1/\log T$. Then, we have
\begin{align}
  \int_T^{2T}{L\left(s\!+\!\alpha,f\right)L\left(1\!-\!s\!+\!\beta,\overline{f}\right)\d t}
  \ = \ \int_T^{2T}&
  \left\{L(1\!+\!\alpha\!+\!\beta,f\!\times\!\overline{f})
    +\left(\frac{t}{2\pi}\right)^{-2(\alpha+\beta)}
    L(1\!-\!\alpha\!-\!\beta,f\!\times\! \overline{f})\right\}\d t
    +O(T)
\end{align}
as $T \to \infty$.
\end{theorem}

The main term of Theorem~\ref{thm:shiftedmoment} verifies
a conjecture arising from the recipe of Conrey, Farmer, Keating, Rubinstein, and Snaith (see~\cite{CFKRS}), which also predicts the additional lower order 
terms. Using other methods, the expected lower order terms in  Theorem \ref{thm:shiftedmoment} can be deduced in the case that $f$ is a holomorphic primitive cusp form of even weight for the full modular group. Farmer \cite{farmer} has proved the asymptotic behavior of the mollified integral second shifted moment of such $L(s,f)$ when the mollifier is a Dirichlet polynomial of length less than $T^{1/6-\varepsilon}$, where $\varepsilon>0$ is small. Recently, Bernard \cite{Bernard} has proved the asymptotic behavior of the smooth mollified shifted second moment of such $L(s,f)$ where the mollifier is a Dirichlet polynomial of length less than $T^{5/27}$, with the additional requirement that the shifts are $o\left(1/\log^2(T)\right)$. (See \cite[Remark 2]{Bernard}.)

To deduce the upper bound on $\mu_L$ appearing in Theorem \ref{thm:smallgaps}, we study the pair correlation of nontrivial zeros of $L(s) \in \selberg$ given by Murty and Perelli \cite{murty-perelli}. We employ an argument of \cite{GGOS} with a new idea of Carniero, Chandee, Littmann, and Milinovich \cite{CCLM} to prove the list of upper bounds given in Theorem~\ref{thm:smallgaps}. 

\smallskip
The article is organized as follows. For primitive $\GL(2)$ $L$-functions, we prove Theorem~\ref{thm:gaps}  on large gaps between zeros of $L(1/2+it,f)$ in Section~\ref{sec:gaps}. We prove Theorem~\ref{thm:mixed} on the mixed second moments of derivatives of $L(1/2+it,f)$ in Section~\ref{mixed} and the proof of Theorem~\ref{thm:shiftedmoment} regarding shifted moments  in Section~\ref{sec:shiftproof}. For primitive $L$-functions in the Selberg class, we prove Theorem~\ref{thm:smallgaps}} on small gaps between zeros of $L(s)$ in Section~\ref{smallgaps}.


\section{Proof of Theorem~\ref{thm:gaps}}\label{sec:gaps}

We now show that Theorem \ref{thm:gaps} follows from Theorem \ref{thm:mixed}. We closely follow the proof of~\cite[Theorem 1]{caroline}, which in turn is a variation of a method of Hall \cite{Hall} using some ideas due to Bredberg \cite{Bredberg}.

Define the function
\begin{equation}\label{eq:testfctn}
g(t) \ := \  e^{i\rho t \log T}L\left(\tfrac{1}{2}\!+\!it,f \right),
\end{equation}
where $\rho$ is a real constant that will be chosen later to optimize our result. Fix a primitive
form, $f$, on $\GL(2)$, so that $f$ corresponds either to a primitive holomorphic cusp form or to a primitive Maass cusp form. Let $\widetilde\gammaf$ denote an
ordinate of a zero of $L(s,f)$ on the (normalized) critical line
$\Re(s)=1/2$. Note that $g(t)$ has the same zeros as
$L\left(1/2+\!it,f\right)$, that is, $g(t)\!=\!0$ if and only if
$t\!=\!\tgammaf$. Let $\left\{ \tgammaf(1),\tgammaf(2),\ldots,\tgammaf(N)\right\}$
denote the set of distinct zeros of $g(t)$ in the interval $[T,2T]$. Let
\begin{equation}
  \kappa_T\ := \  \max\left\{\tgammaf(n\!+\!1)-\tgammaf(n): T+1\leq\tgammaf(n) \leq 2T-1\right\},
\end{equation}
and note that $\Lambda_f\geq\limsup_{T\ra\infty}\kappa_T$. Without loss of
generality, we may assume that
\begin{equation}\label{eq:zeroassumption}
\tgammaf(1)-T\ll1\quad\text{and}\quad 2T-\tgammaf(N)\ll1,
\end{equation}
as otherwise there exist zeros $\tgammaf(0)\leq\tgammaf(1)$ and
$\tgammaf(N+1)\geq\tgammaf(N)$ such that $\tgammaf(0)-\tgammaf(1)$
and $\tgammaf(N+1)-\tgammaf(N)$ are $\gg1$, and the theorem holds for
this reason.

The following lemma, due to Bredberg \cite[Corollary 1]{Bredberg}, is a variation of Wirtinger's inequality \cite[Theorem 258]{wirtinger}.
\begin{lemma}[Wirtinger's inequality]\label{lem:wirtinger}
  Let $y : [a,b]\ra\C$ be a continuously differentiable function, and suppose
  that $y(a)=y(b)=0$. Then
  \begin{equation}
    \int_a^b|y(x)|^2\d x\ \leq\ \left(\frac{b-a}\pi\right)^2\int_a^b|y'(x)|^2\d x.
  \end{equation}
\end{lemma}
Let $\varepsilon>0$ be small. By the definition of $\kappa_T$ and Lemma \ref{lem:wirtinger}, for each pair of
consecutive zeros of $g(t)$ in the interval $[T,2T]$ we have
\begin{equation}\label{eq:wirtapp}
  \int_{\tgammaf(n)}^{\tgammaf(n+1)}|g(t)|^2\d t\ \leq\ \frac{\kappa_T^2}{\pi^2}
  \int_{\tgammaf(n)}^{\tgammaf(n+1)}|g'(t)|^2\d t.
\end{equation}
Upon summing both sides of the equation in~(\ref{eq:wirtapp}) for
$n=1,2,\ldots,N-1$, we have
\begin{equation}\label{eq:wirtinterval}
  \int_{\tgammaf(1)}^{\tgammaf(N)}|g(t)|^2\d t\ \leq\ \frac{\kappa_T^2}{\pi^2}
  \int_{\tgammaf(1)}^{\tgammaf(N)}|g'(t)|^2\d t.
\end{equation}
Subconvexity bounds for primitive $\GL(2)$ $L$-functions along the critical line yield
$|g(t)| \ll |t|^{1/2 - \delta}$, where $\delta>0$ is a fixed constant.
Good~\cite{goodsubconvex} established subconvexity in the $t$-aspect for holomorphic forms
of full level, achieving $\delta < 1/6$. Meurman~\cite{meur} achieved the same for
Maass forms of full level. Jutila and Motohashi~\cite{jutila-motohashi} obtained a hybrid
bound in the $t$- and eigenvalue aspects for full-level holomorphic and Maass forms.
Blomer and Harcos~\cite{blomer-harcos} obtained $\delta < 25/292$ for
holomorphic and Maass forms of arbitrary level and nebentypus in the $t$-aspect.
Michel and Venkatesh~\cite{mvsubconvexity} proved general subconvexity for $\GL(2)$.
By ~(\ref{eq:zeroassumption}), we have
\begin{equation}\label{eq:fullinterval}
  \int_T^{2T}|g(t)|^2\d t \leq \frac{\kappa_T^2}{\pi^2}\int_{T}^{2T}|g'(t)|^2\d t +
  O\left(T^{1 - 2\delta}\right).
\end{equation}
Observing that $|g(t)|^2=\left|L\left(1/2+\!it\right)\right|^2$ and
\begin{equation}
  |g'(t)|^2\ = \ \left|L'\left(\tfrac12\!+\!it,f \right)\right|^2 + \rho^2\log^2T
  \left|L\left(\tfrac12\!+\!it,f \right)\right|^2 + 2\rho\log T\cdot
  \Re\left(L'\left(\tfrac12\!+\!it,f \right)
    \overline{L\left(\tfrac12\!+\!it,f\right)}\right),
\end{equation}
Theorem~\ref{thm:mixed} implies that
\begin{align}
  \int_T^{2T}\left|L\left(\tfrac12\!+\!it,f \right)\right|^2\d t &\ = \  2c_f T \log T + O(T), \notag\\
  \int_T^{2T}L'\left(\tfrac12\!+\!it,f \right)\overline{L\left(\tfrac12\!+\!it,f \right)}\d t &\ = \  -2c_f T\left(\log T\right)^2+O\left(T\log T\right),
\intertext{and}
  \int_T^{2T}\left|L'\left(\tfrac12\!+\!it,f \right)\right|^2\d t &\ = \  \frac83 c_f T \left(\log T\right)^3+O\left(T\left(\log T\right)^2\right),
\end{align}
where $c_f$ denotes the residue of the simple pole of $L(s,f\!\times\! \overline{f})$ at $s\!=\!1$. Combining
these estimates and noting that
\begin{equation}
\frac{1}{1+O\left(\frac1{c_f}\left(\log T\right)^{-1}\right)} \ = \  1+ O\left(\frac{1}{c_f}\left(\log T\right)^{-1}\right),
\end{equation}
we find that

\begin{equation}\label{eq:almostdone}
  \frac{\kappa_T^2}{\pi^2} \geq \frac3{3\rho^2-6\rho+4}\left(\log T\right)^{-2}
  \left(1 + O\left(\frac1{c_f}(\log T)^{-1}\right)\right).
\end{equation}
The polynomial $3\rho^2-6\rho+4$ is minimized by $\rho\!=\!1$. Therefore, inserting this choice of $\rho$ in \eqref{eq:almostdone}, we obtain
\begin{equation}
  \kappa_T\ \geq\ \frac{\sqrt3\pi}{\log T}
  \left(1 + O\left(\frac1{c_f}\left(\log T\right)^{-1}\right)\right).
\end{equation}


\section{Proof of Theorem~\ref{thm:mixed}}\label{mixed}
In this section we prove Theorem~\ref{thm:mixed}. As in the previous section, we closely follow an argument of ~\cite[Theorem 3]{caroline}, and we include the details here for completeness.

First, note that for $t\in\left[T,2T\right]$, we have
\begin{equation}
  \left(\frac{t}{2\pi}\right)^{-2(\alpha + \beta)} \ = \  T^{-2(\alpha + \beta)}
  \left(1 + O\left(\frac{1}{\log{T}}\right)\right)
\end{equation}
as $T\to\infty$. Also, we have
\begin{equation}
L(1 \pm \alpha \pm \beta,f\!\times\!\overline{f}) \ = \   \frac{\pm c_f}{\alpha\!+\!\beta} + O(1),
\end{equation}
and
\begin{equation}
T^{-2(\alpha+\beta)} \ = \  \sum_{n=0}^{\infty}\frac{(-1)^n2^n(\alpha\!+\!\beta)^n(\log T)^{n}}{n!}.
\end{equation}
Thus, by Theorem \ref{thm:shiftedmoment}, we have
\begin{equation}
 \int_{T}^{2T}L\left(\tfrac{1}{2} \!+\! \alpha \!+\! it, f \right)
  L\left(\tfrac{1}{2} \!+\! \beta \!-\! it, \overline{f}\right)\d t \ = \  F(\alpha \!+\! \beta; T) + O(T),
\end{equation}
and
\begin{equation}
  F(\alpha \!+\! \beta; T) \ := \  c_{f}T
  \sum_{n \geq 0}\frac{(-1)^{n}2^{n+1}(\alpha \!+\! \beta)^{n}(\log{T})^{n+1}}{(n\!+\!1)!}.
\end{equation}
Let
\begin{equation}
  R(\alpha,\beta; T) \ := \  \int_{0}^{T}L\left(\tfrac{1}{2} \!+\! it \!+\! \alpha, f\right)
  L\left(\tfrac{1}{2} \!-\! it \!+\! \beta, \overline{f}\right)\d t - F(\alpha \!+\! \beta; T).
\end{equation}
Then $R(\alpha,\beta; T)$ is an analytic function of two complex variables
$\alpha, \beta$ for $\Re(\alpha), \Re(\beta)<1/2$;
moreover, if $|\alpha|, |\beta| \ll 1/\log T$, then Theorem~\ref{thm:shiftedmoment} implies that
\begin{equation}\label{eq:R-error}
R(\alpha,\beta;T) \ = \  O(T)
\end{equation}
as $T\to \infty$. Differentiating, we find
\begin{equation}\label{eq:R-diff}
  \int_{T}^{2T}L^{(\mu)}\left(\tfrac{1}{2}\!+\!it\!+\!\alpha, f\right)
  L^{(\nu)}\left(\tfrac{1}{2}-it\!+\!\beta, \overline{f} \right)\d t
  \ = \  \frac{\partial^{\mu+\nu}F(\alpha\!+\!\beta;T)}
  {\partial\alpha^{\mu}\partial\beta^{\nu}} + R_{\mu,\nu}(\alpha,\beta;T),
\end{equation}
where $\mu$ and $\nu$ are fixed nonnegative integers and
\begin{equation}
  R_{\mu,\nu}(\alpha,\beta;T) \ := \  \frac{\partial^{\mu+\nu}R(\alpha,\beta;T)}
  {\partial \alpha^{\mu} \partial \beta^{\nu}}.
\end{equation}
Let $\Omega = \left\{\omega\in\C : |\omega-\alpha| = 1/\log T\right\}$.
By contour integration and Cauchy's integral formula, \eqref{eq:R-error} implies that
\begin{align}\label{eq:R-Cauchy}
  \frac{\partial^{\mu}}{\partial \alpha^{\mu}} R(\alpha,\beta;T)
  \ = \  \frac{\mu !}{2 \pi i}
  \int_{\Omega}\frac{R(\omega,\beta;T)}{(\omega - \alpha)^{\mu+1}}\d \omega
  &\ = \  O\left(\mu ! T(\log{T})^{\mu}\right).
\end{align}
A second application of Cauchy's integral formula yields
\begin{equation}\label{eq:R-Cauchy2}
  R_{\mu,\nu}(\alpha,\beta;T)
  \ := \  \frac{\partial^{\mu+\nu}}{\partial \alpha^{\mu} \partial \beta^{\nu}}
  R(\alpha,\beta;T) \ = \  O\left(\mu ! \nu ! T(\log{T})^{\mu + \nu}\right).
\end{equation}
Setting $\alpha\!=\!\beta \!=\! 0$, we obtain
\begin{equation}\label{eq:R-diff2}
  \int_{T}^{2T}L^{(\mu)}\left(\tfrac{1}{2} \!+\! it, f \right)
  L^{(\nu)}\left(\tfrac{1}{2} \!-\! it, \overline{f} \right)\d t
  \ = \  \left[\frac{\partial^{\mu+\nu}F(\alpha\!+\!\beta;T)}
    {\partial \alpha^{\mu} \partial \beta^{\nu}} \right]_{\alpha=\beta=0}
  + O\left(\mu ! \nu ! T(\log{T})^{\mu + \nu}\right).
\end{equation}
Differentiating $F(\alpha \!+\! \beta; T)$ with respect to $\alpha$ and $\beta$, we find
\begin{equation}\label{eq:F-diff}
  \left[\frac{\partial^{\mu + \nu} F(\alpha \!+\! \beta; T)}
    {\partial \alpha^{\mu} \partial \beta^{\nu}} \right]_{\alpha = \beta = 0}
  \ = \  c_f T\frac{(-1)^{\mu+\nu}2^{\mu+\nu+1}(\log{T})^{\mu+\nu+1}}{\mu \!+\! \nu \!+\! 1}.
\end{equation}
Inserting \eqref{eq:F-diff} in \eqref{eq:R-diff2}, the theorem now follows by summing over the dyadic intervals $[T/2,T]$, $[T/4,T/2]$, $[T/8,T/4], \ldots$.


\section{Properties of $L$-functions}\label{properties}
In this section we collect some basic facts about the $L$-functions under consideration. The $L$-functions treated by
Theorem~\ref{thm:gaps} have an automorphic characterization; they are
associated to primitive cusp forms on $\GL(2)$ over $\Q$. We
summarize some of their properties below. On the other hand, we prove
Theorem~\ref{thm:smallgaps} for the large set of $L$-functions in the
analytic axiomatic classification of $L$-functions due to
Selberg~\cite{selbergclass}, which we denote $\selberg$;
see Conrey and Ghosh~\cite{conreyghosh}, Murty~\cite{murty}, and
Kaczorowski and Perelli~\cite{kaczorowskiperelli} for the basic
properties of $\selberg$. 

Though it is certain that the primitive automorphic cuspidal $\GL(2)$
$L$-functions we consider are members of the Selberg class, and the
membership of primitive holomorphic cusp forms has been established,
it has not yet been established in the real-analytic case of Maass cusp
forms. Most notably, the Ramanujan hypothesis for Maass forms is not yet
a theorem. Fortunately, we do not need to assume the Ramanujan hypothesis (or
the Generalized Riemann Hypothesis) for Theorem~\ref{thm:gaps}. Likewise, there are some specific
assumptions we do need for Theorem~\ref{thm:gaps} that we do not need
for Theorem~\ref{thm:smallgaps}.

Therefore, the two sets of assumptions we need for our two main theorems
have nontrivial intersection but one set is not a proper subset of the
other. The approach we take to presenting these two sets of hypotheses is as follows.
First, we will present the axioms of the Selberg class $\selberg$ needed for
Theorem~\ref{thm:smallgaps}. Then, we give information specific to the automorphic
$\GL(2)$ $L$-functions to which Theorem~\ref{thm:gaps} applies.
\subsection{The Selberg class $\selberg$}\label{sec:selberg}
For our purposes, the following axiomatic definition is sufficient, and we 
follow~\cite{murty-perelli} in our presentation.
\begin{enumerate}[topsep=0in,label=(\textit{\roman*})]
\item\label{selberg:dirichlet} (\textit{Dirichlet series})\quad
  Every $L \in\pazocal S$ is a Dirichlet series
  \be L(s)\ = \ \sum_{n=1}^\infty \frac{a_L(n)}{n^{s}}\ee absolutely convergent for
  $\Re(s)>1$.
\item\label{selberg:continuation} (\textit{Analytic continuation})\quad
  There exists an integer $a\geq0$
  such that $(s-1)^aL(s)$ is an entire function of finite order.
\item\label{selberg:funeq} (\textit{Functional equation})\quad
  Every $L \in\selberg$ satisfies a functional equation of type
  \be L_\infty(s)L(s)\ =: \ \Lambda(s)\ = \ \epsilon\overline\Lambda(1-s),\ee
  where $\overline\Lambda(s)$ denotes $\overline{\Lambda(\overline s)}$, and
  \be L_\infty\ = \ Q^s\prod_{j=1}^r\Gamma(w_j s+\mu_j)\ee
  with $Q>0,w_j>0,\Re(\mu_j)\geq0$, and $|\epsilon|=1$. The \textit{degree} of $L$ is given
  by
  \be\label{degree}
  m_L\ :=\ 2\sum_{j=1}^{r}w_j.
  \ee
\item\label{selberg:rama} (\textit{Ramanujan hypothesis})\quad
  For all positive integers $n$, we have $a_L(n)\ll n^{o(1)}$.
\item\label{selberg:euler} (\textit{Euler product})\quad
  Every $L \in\selberg$ satisfies
  \be\log L(s)\ = \ \sum_{n=1}^\infty \frac{b_L(n)}{n^{s}},\ee
  where $b_L(n)=0$ unless $n=p^\ell$ for some $\ell\geq 1$, and $b_L(n)\ll n^\theta$
  for some $\theta<1/2$.
\end{enumerate}

\subsection{Properties of $\GL(2)$ $L$-functions}
In this section we collect some basic facts and hypotheses concerning $L$-functions attached to primitive (holomorphic or Maass) cusp forms on $\GL(2)$. We begin by combining some general remarks from~\cite[Section 1.1]{CFKRS}, \cite[Section 2]{rudnicksarnak}, and \cite[Section 3.6]{rubinstein}. For a more in-depth study of these $L$-functions, we refer the reader to~\cite[Chapter 5]{IK}.

Let $f$ be a primitive (holomorphic or Maass) cusp form on $\GL(2)$ over $\Q$ with level $q$.
For $\Re(s)>1$, let
\be L(s,f) \ := \  \sum_{n=1}^\infty \frac{a_f(n)}{n^s}
\ = \ \prod_p\left(1-\frac{\alpha_f(p)}{p^s}\right)^{-1}
\left(1-\frac{\beta_f(p)}{p^s}\right)^{-1}\ee
be the global $L$-function attached to $f$ (as defined by Godement and Jacquet
in~\cite{GJ} and Jacquet and Shalika in~\cite{JS}), where the Dirichlet
coefficients $a_f(n)$ have been normalized so that $\Re(s)=1/2$ is the critical
line of $L(s,f)$. The numbers $\alpha_f,\beta_f$ are called the non-archimedean
Satake or Langlands parameters. We assume $L(s,f)$ is
primitive; that is, we assume $L(s,f)$ cannot be written as the product of two degree 1 $L$-functions. Then $L(s,f)$ admits an analytic continuation to an entire function of order 1. Additionally, there is a root number $\ep_f\in\C$ with $|\ep_f|=1$
and a function $L_\infty(s,f)$ of the form
\be\label{eq:linfty}
L_\infty(s,f)\ = \ P(s)Q^s\Gamma(w_js+\mu_1)\Gamma(w_js+\mu_2),\ee
where $Q>0$, $w_j>0$, $\Re(\mu_j)\geq0$, and $P$ is a polynomial whose only
zeros in $\sigma>0$ are the poles of $L(s)$, such that the completed
$L$-function
\be \Lambda(s,f) \ := \  L_\infty(s,f)L(s,f) \ee
is entire, and
\be \Lambda(s,f)\ = \ \ep_f \Lambda(1-s,\overline f), \ee
where $\overline f(z)=\overline{f(\overline z)}$,
$\Lambda(s,\overline f)=\overline{\Lambda}(s,f)=\overline{\Lambda(\overline s,f)}$, etc.
It will be convenient to write this functional equation in asymmetric form
\be\label{eq:funeqasymmetric}
L(s,f)\ = \ \ep_f \Phi_f(s) L(1-s,\overline f), \ee
where $\Phi_f(s)=L_\infty(1-s,\overline f)/L_\infty(s,f)$.

In light of this, axioms \ref{selberg:dirichlet}, \ref{selberg:continuation},
and \ref{selberg:funeq} of $\selberg$ for $L(s,f)$ are
satisfied.\footnote{The condition that
  $\Re(\mu_j)\geq0$ in axiom \ref{selberg:funeq} is not proven in the case of
  Maass cusp forms, though it is conjectured to hold, and is immaterial to the proof of Theorem~\ref{thm:gaps}. See~\cite[p.\,13]{bumptrace} and~\cite[\S1.5]{sarnakautomorphic}.}
We now isolate the four properties on $L(s,f)$ for $f$ a primitive cusp form on
$GL(2)$ over $\Q$ that we will make use of in our proof of
Theorem~\ref{thm:gaps}. All of these properties are known for such $f$;
that is, none are conjectural.
\begin{Hlist}
\item \label{hyp:entire} $L(s,f)$ is entire.
\item \label{hyp:funeq} $L(s,f)$ satisfies a functional equation
  of the special form
  \be
  \Lambda(s,f)\ := \ L_\infty(s,f)L(s,f)\ = \ \ep_f\Lambda(1-s,\overline f),
  \ee
  where
  \be L_\infty(s,f)
  \ = \ Q^s\Gamma\left(\tfrac12s+\mu_1\right)\Gamma\left(\tfrac12s+\mu_2\right), \ee
  with $\{\mu_j\}$ stable under complex conjugation and the other notation
  the same as in~\eqref{eq:linfty}. Note that in the
  notation of~(\ref{eq:linfty}), $w_j=1/2$, which is conjectured to hold for
  arithmetic $L$-functions. The numbers $\mu_1,\mu_2$ are called the archimedean
  Langlands parameters.
\item \label{hyp:convolution}
  The convolution Dirichlet series given by
  \be
  L(s,f\times\overline f)\ :=\ \sum_{n=1}^\infty\frac{\left|a_f(n)\right|^2}{n^s},\qquad\Re(s)>1
  \ee
  is an $L$-function whose analytic continuation has a simple pole at $s=1$.
  (This is conjectured to be equivalent to $L(s,f)$ being a primitive
  $L$-function.) We denote the residue of this simple pole by $c_f$.
\item \label{hyp:coeffsquaredsum}
  For sufficiently large $X$,
  $\sum_{n\leq X}\left|a_f(n)\right|^2 \ll X$.
\end{Hlist}
Note that \ref{hyp:entire} is true by general arguments for $L$-functions associated to
primitive cupsidal automorphic representations on $\GL(2)$.
In the case that $f$ is holomorphic of weight $k$ and level $q$, the Dirichlet
series associated to $L(s,f)$ is formed from the (normalized) coefficients
$\lambda_f(n)$ in the Fourier expansion
\be f(z)\ = \ \sum_{n=1}^\infty\lambda_f(n)n^{(k-1)/2}e(nz), \ee
which satisfy the Deligne divisor bound
$\left|\lambda_f(n)\right|\leq d(n)$. Hence, axiom~\ref{selberg:rama} is
satisfied for $f$, though this is immaterial to our arguments towards
Theorem~\ref{thm:gaps}. For $L(s,f)$ attached to such holomorphic
$f$, \ref{hyp:funeq} holds for $Q=\pi^{-1}$, $\mu_1=(k-1)/2$ and
$\mu_2=(k+1)/2$.
In the case that $f$ is instead a Maass cusp form, ~\ref{hyp:funeq} is known
to hold for $Q=\pi^{-1}$ and complex $\mu_j$. We only require the known condition that
the set $\{\mu_j\}$ is stable under complex conjugation.

We note that ~\ref{hyp:convolution} was established by the work of Rankin and Selberg for
Hecke (holomorphic) cusp forms; it is established in generality far exceeding
our needs by Mœglin and Waldspurger~\cite{MW}; see~\cite[\S5.11--5.12]{IK} for
details.

For holomorphic $f$, ~\ref{hyp:coeffsquaredsum} holds. Actually, it is known that
\be\sum_{n\leq x}\left|a_f(n)\right|^2 = c_f x + o(x);\ee
see~\cite[Proposition 5.1]{micah-nathan}.
For real-analytic $f$, ~\ref{hyp:coeffsquaredsum} holds;
see~\cite[\S3.2]{harcos}
and~\cite[Theorem 3.2, (8.7), and (9.34)]{iwaniecspectral}.

Last, we introduce the notion of the analytic conductor $\mathfrak q(s,f)$ for
an automorphic $L$-function attached to a cusp form of level $q$ on $\GL(2)$.
In this case, specializing Harcos~\cite{harcos}, who follows~\cite{iwaniecsarnak},
\be \mathfrak q(s,f) \ := \ \frac{q}{(2\pi)^2}
\left|s+2\mu_1\right|\left|s+2\mu_2\right|.\ee

\section{Proof of Theorem  \ref{thm:smallgaps}}\label{smallgaps}
Let $L(s)$ be a primitive $L$-function in the Selberg class $\selberg$. Recall from \eqref{eq:numLzeros} and \eqref{eq:muLdef} the definitions
\[
N(T,L)\ := \ \sum_{0<\gammaL \leq T} 1
\]
and
\[
\mu_L\ := \  \liminf_{n\to \infty}\,\frac{\gammaL(n\!+\!1)-\gammaL(n)}{N(T,L)},
\]
where $\gammaL(n)$ denotes the $n$th nontrivial zero of $L(s)$. In this section, assuming the Generalized Riemann Hypothesis for $L(s)$, we prove an upper bound for $\mu_L$ using ideas first introduced by Montgomery~\cite{montgomerypc} to study the pair correlation of zeros of $\zeta(s)$. 

Assuming the Riemann Hypothesis and writing the nontrivial zeros of $\zeta(s)$ as $1/2+\gamma$, if $0\notin [\alpha,\beta]$ and $T\ra\infty$, the pair correlation conjecture is the statement that
\be\#\left\{0<\gamma,\gamma'<T:
  \alpha\leq\frac{(\gamma-\gamma')\log T}{2\pi}\leq\beta\right\}
\sim\left(\frac T {2\pi}\log T\right)\int_\alpha^\beta
\left(1-\left(\frac{\sin \pi u}{\pi u}\right)^2\right)\d u, \ee
which is consistent with the pair correlation function of eigenvalues of random
Hermitian matrices. Montgomery originally formulated his pair correlation conjecture for $\zeta(s)$,
but strong evidence has accumulated since he made his conjecture to suggest that
the nontrivial zeros of any general primitive $L$-function share the same
statistics as eigenvalues of matrices chosen randomly from a matrix ensemble
appropriate to the $L$-function; this philosophy is articulated by Katz and
Sarnak in~\cite{katzsarnakbook} and~\cite{katzsarnakarticle}.
In view of this, Montgomery's pair correlation
conjecture has been generalized to all $L$-functions in the Selberg class
$\selberg$.

Montgomery proved that $\zeta(s)$ satisfies his pair correlation hypothesis for
restricted support, and Murty and Perelli~\cite{murty-perelli} proved a general
version for all primitive $L$-functions in the Selberg class for restricted
support inversely proportional to the degree of the function. We use the
pair-correlation for zeros of primitive $L\in\pazocal S$ to establish an upper
bound on small gaps of primitive $L$-functions in the Selberg class.

Following~\cite{montgomerypc} and \cite{murty-perelli}, we let $\gammaL, \gammaL'$ denote ordinates of nontrivial zeros of $L(s)$ and put
\be\label{eq:falphadef}
F_L(\alpha)\ = \ F_L(\alpha,T)\ := \ \left(\frac{m_LT}{2\pi} \log T\right)^{-1}
\sum_{0<\gammaL,\gammaL'\leq T}
T^{im_L\alpha(\gammaL-\gammaL')}w(\gammaL-\gammaL'),\ee
where $w(u) =  4 /(4+u^2)$ and $m_L$ is the degree of $L$ as defined in \eqref{degree}. The pair correlation conjecture then states that as $T\ra\infty$, we have
\be F(\alpha)\ = \
\begin{cases}
  |\alpha|+m_LT^{-2|\alpha|m_L}\log T(1+o(1))+o(1),&\text{if }|\alpha|\leq 1, \\
  1+o(1),&\text{if }|\alpha|\geq 1,
\end{cases}\ee
uniformly for $\alpha$ in any bounded interval. Define the functions
\be \upsilon_L(n)\ := \ b_L(n)\log n
\ee
and
\be \label{eq:upsilondef} \upsilon_L(n,x)\ := \ \begin{cases}
 \displaystyle \upsilon_L(n)\left(\frac n x\right)^{1/2},&n\leq x, \\[.15in]
 \displaystyle \upsilon_L(n)\left(\frac x n\right)^{3/2},&n>x.
\end{cases}\ee
Murty and Perelli~\cite{murty-perelli} have proved the following result for $F_L(\alpha)$. 
\begin{prop}\label{prop:murtyperelli}
  With $L(s)\in\selberg$ as above and assuming the Generalized Riemann Hypothesis, let $\vep>0$ and
  $x=T^{\alpha m}$. Then, uniformly for $0\leq\alpha\leq(1-\vep)/m_L$ as
  $T\ra\infty$, we have
  \be F_L(\alpha)\ = \ \frac1{m_Lx\log T}\sum_{n=1}^\infty\upsilon(n,x)
  \overline{\upsilon(n,x)}+m_LT^{-2\alpha m_L}\log T(1+o(1))+o(1). \ee
\end{prop}
Under the assumptions of Proposition~\ref{prop:murtyperelli} and Hypothesis~\ref{eq:selbergconj},  we may rewrite $F_L(\alpha)$ via partial summation (c.f.~\cite[\S 4]{murty-perelli} for details) as
\be\label{eq:falphapc} F_L(\alpha)\ = \ \alpha+m_LT^{-2\alpha m_L}\log T(1+o(1))+o(1).\ee

For the application to small gaps, we generalize an argument of
Goldston, Gonek, \"Ozl\"uk, and Snyder in~\cite{GGOS}, with a new modification  
of Carneiro, Chandee, Littmann, and Milinovich in~\cite{CCLM}.

We use the function $F_L(\alpha)$ to evaluate sums over differences of zeros. We first record a convolution formula involving $F_L(\alpha)$ that will prove to be useful. Let $r(u)\in L^1$, and define the Fourier transform by
\be\hat r(\alpha)\ = \ \int_{-\infty}^\infty r(u)e(\alpha u)\d u.\ee
If $\hat r(u)\in L^1$, we have almost everywhere that
\be r(u)\ = \ \int_{\infty}^\infty \hat r(\alpha)e(-u\alpha)\d\alpha.\ee
Multiplying~\eqref{eq:falphadef} by $\hat r(\alpha)$ and integrating, we obtain
\be\label{eq:falphaconvolution}
\sum_{0<\gammaL,\gammaL'\leq T}
r\left(\left(\gammaL-\gammaL'\right)\frac{m_L\log T}{2\pi}\right)
w(\gammaL-\gammaL')
\ = \
\left(\frac{m_LT}{2\pi} \log T\right)
\int_{-\infty}^\infty \hat r(\alpha)F_L(\alpha)\d\alpha.
\ee
We make use of the following bound for $F_L(\alpha)$.
\begin{lemma}\label{lem:falphaest}
  Assume the Generalized Riemann Hypothesis, and let $A>1$ be fixed. Then, as $T\ra\infty$, we have
  \be\int_{1/m_L}^\xi(\xi-\alpha)F_L(\alpha)\d\alpha\ \geq\ 
  \frac{\xi^2}2-\frac\xi2\left(1+\frac1{m_L^2}\right)+\frac1{3m_L^3}+o(1)\ee
  uniformly for $1\leq\xi\leq A$.
\end{lemma}
\begin{proof}[Proof of Lemma~\ref{lem:falphaest}]
  Following~\cite{CCLM}, we start with the Fourier pair
  \be r_\xi(u)\ = \ \left(\frac{\sin\pi\xi u}{\pi\xi u}\right)^2\quad\text{and}\quad
  \hat r_\xi(\alpha)\ = \ \frac1{\xi^2}\max\left(\xi-|\alpha|,0\right).\ee
  Let $\mathfrak{m}_\rho$ denote the multiplicity of the zero with generic ordinate
  $1/2+\gamma_L$. By trivially replacing the count of zeros up to height $T$ (
  counted with multiplicity) with the diagonal of a weighted sum of $r_\xi$
  evaluated at differences in zeros and using the convolution
  formula~\eqref{eq:falphaconvolution}, we obtain
  \be\label{eq:xitrick}\begin{aligned}
    1+o(1)\ &\leq\ \left(\frac{m_LT}{2\pi}\log T\right)^{-1}\sum_{0<\gammaL\leq T}\mathfrak{m}_\rho\\
    &\leq\ \left(\frac{m_LT}{2\pi}\log T\right)^{-1}\sum_{0<\gammaL,\gammaL'\leq T}
    r_\xi\left(\left(\gammaL'-\gammaL\right)\frac{m_L\log T}{2\pi}\right)
    w\left(\gammaL-\gammaL'\right) \\
    &\ = \ \int_{-\xi}^\xi \hat r(\alpha)F(\alpha)\d\alpha.
  \end{aligned}\ee
  Leveraging~\eqref{eq:falphapc}, the evenness of the integrand allows us to
  write
  \be\label{eq:xiestimate}\begin{aligned}
    \int_{-\xi}^\xi \hat r(\alpha)F_L(\alpha)\d\alpha
    &\ = \ \frac2{\xi^2}\int_0^{1/m_L}(\xi-\alpha)F_L(\alpha)\d\alpha
    +\frac2{\xi^2}\int_{1/m_L}^\xi (\xi-\alpha)F_L(\alpha)\d\alpha \\
    &\ = \ \frac1{\xi}\left(1+\frac1{m_L^2}\right)-\frac23\frac1{\xi^2m_L^3}
    +\frac2{\xi^2}\int_{1/m_L}^\xi (\xi-\alpha)F_L(\alpha)\d\alpha+o(1)
  \end{aligned}\ee
  uniformly for $1\leq\xi\leq A$. Inserting~\eqref{eq:xiestimate}
  into~\eqref{eq:xitrick} establishes the lemma.
\end{proof}
We now prove the following result.
\begin{theorem}\label{thm:smallgaps1}
  Let $L(s)\in\selberg$ be primitive of degree $m_L$.
  Assume the Generalized Riemann Hypothesis and Hypothesis~\ref{eq:selbergconj}.
  Then, with $\kappa_{m_L}$ as in~\eqref{eq:kappadef} as $T\ra\infty$, we have
  \be\label{sum}\sum_{0<\gammaL-\gammaL'\leq\frac{2\pi\lambda}{m_L\log T}}1
  \geq\left(\frac12-\vep\right)\frac{m_LT}{2\pi}\log T
  \begin{aligned}[t]
    \Bigg(
    &\lambda-1+2\lambda\int_{0}^{1/m_L} \left(1-|\lambda\alpha|+\frac{\sin 2\pi|\lambda\alpha|}{2\pi}\right)\alpha\d\alpha \\
    &-4\pi\lambda^3\int_{\kappa_{m_L}}^{1/\lambda}\sin(2\pi\lambda\alpha)
    \left(\frac{\alpha^2}2-\frac\alpha2\left(1+\frac1{m_L^2}\right)
      +\frac1{3m_L^3}+o(1)\right)\d\alpha\Bigg).\end{aligned}\ee
\end{theorem}

The upper bounds on $\mu_L$ given in Theorem~\ref{thm:smallgaps} now follow upon
straightforward numerical computation of the first positive value of $\lambda$
for which the right side of the inequality in \eqref{sum} in the theorem becomes positive. 

\begin{proof}[Proof of Theorem~\ref{thm:smallgaps1}]
  Consider the Fourier pair
  \begin{equation}
    h(u)\ = \ \left(\frac{\sin \pi u}{\pi u}\right)^2\left(\frac1{1-x^2}\right) \qquad
    \text{and} \qquad
    \hat h(\alpha)\ = \ \max\left(1-|\alpha|+\frac{\sin 2\pi|\alpha|}{2\pi},0\right).
  \end{equation}
  Here $h(u)$ is the Selberg minorant of the charactaristic function of the interval
  $[-1,1]$ in the class of functions with Fourier transforms with support in
  $[-1,1]$. Take $r(u)=h(u/\lambda)$. Then $r(u)$ is a minorant of the
  characteristic function on $[-\lambda,\lambda]$, and
  $\hat r(\alpha)=\lambda\hat h(\lambda \alpha)$. By~\eqref{eq:falphaconvolution}, this allows us to write
  
  \be\begin{aligned}
    \sum_{0<\gammaL\leq T}\mathfrak{m}_\rho
    +2\!\!\!\!\!\!\!\!\!\!\!\!\!\!\!\!\!\sum_{0<\gammaL-\gammaL'\leq\frac{2\pi\lambda}{m_L\log T}}1\ 
    &\geq\ 
    \sum_{0<\gammaL,\gammaL'\leq T}
    r\left(\left(\gammaL-\gammaL'\right)\frac{m_L\log T}{2\pi}\right)
    w(\gammaL-\gammaL') \\
    &\ = \ \left(\frac{m_LT}{2\pi} \log T\right)
    \int_{-1/\lambda}^{1/\lambda} \hat r(\alpha)F_L(\alpha)\d\alpha.
  \end{aligned}\ee
  We may take
  \be\sum_{0<\gammaL\leq T}\mathfrak{m}_\rho\ \sim\ N(T)\ \sim\ \frac{m_LT}{2\pi}\log T,\ee
  for else the theorem is trivially true. Upon inserting~\eqref{eq:falphapc} in its domain
  of validity and estimating the integral arising from the $m_LT^{-2\alpha m_L}\log T$
  portion trivially, we find
  \be\label{eq:intermediateestimate}
  \sum_{0<\gammaL-\gammaL'\leq\frac{2\pi\lambda}{m_L\log T}}1
  \ = \ \left(\frac12-\vep\right)\frac{m_LT}{2\pi}\log T
  \left(\lambda-1+2\lambda\int_{0}^{1/m_L} \hat h(\lambda\alpha)\alpha\d\alpha
    +2\lambda\int_{1/m_L}^{1/\lambda} \hat h(\lambda\alpha)F_L(\alpha)\d\alpha
  \right).\ee
  Next, still following~\cite{CCLM} and~\cite{goldstontrick}, we define the
  function
  \be I(\xi)\ = \ \int_{1/m_L}^\xi(\xi-\alpha)F_L(\alpha)\d\alpha.\ee
  Note that $I(\xi)$ enjoys the properties
  \be I'(\xi)\ = \ \int_{1/m_L}^\xi F_L(\alpha)\d\alpha\quad\text{and}\quad
  I''(\xi)\ = \ F_L(\xi).\ee Integrating by parts twice and observing that
  $\hat h(1)=\hat h'(1)=0$ allow us to write
  \be\label{eq:intbypartstrick}
  \int_{1/m_L}^{1/\lambda}\hat h(\lambda\alpha)F_L(\alpha)\d\alpha
  \ = \ \int_{1/m_L}^{1/\lambda}\hat h(\lambda\alpha)I''(\alpha)\d\alpha
  \ = \ \lambda^2\int_{1/m_L}^{1/\lambda}\hat h''(\lambda\alpha)I(\alpha)\d\alpha.\ee
  Provided $\alpha\geq0$, $\hat h''(\lambda\alpha)=-2\pi\sin(2\pi\lambda\alpha)$,
  which is non-negative for $1\leq\alpha\leq1/\lambda$ if $1/2\leq\beta\leq1$.
  Moreover, as $F_L(\alpha)$ is positive, Lemma~\ref{lem:falphaest} provides
  a nontrivial bound for
  \be\label{eq:kappadef}\xi\ \geq\ 
  \frac12\left(1+\frac1{m_L^2}+\frac{\sqrt{3-8m_L+6m_L^2+3m_L^4}}{m_L^2\sqrt 3}\right)
  \ := \ \kappa_{m_L}.\ee
  Since $m_L$ is a positive integer, we always have $\kappa_{m_L}>1$.
  Hence, inserting the estimate from Lemma~\ref{lem:falphaest}
  into~\eqref{eq:intbypartstrick}, we find that
  \be\label{eq:tailintest}
  \int_{1/m_L}^{1/\lambda}\hat h(\lambda\alpha)F_L(\alpha)\d\alpha\ 
  \geq\ -2\pi\lambda^2\int_{\kappa_{m_L}}^{1/\lambda}\sin(2\pi\lambda\alpha)
  \left(\frac{\alpha^2}2-\frac\alpha2\left(1+\frac1{m_L^2}\right)
    +\frac1{3m_L^3}+o(1)\right)\d\alpha\ee
  for $1/2\leq\beta\leq1$. Inserting~\eqref{eq:tailintest}
  into~\eqref{eq:intermediateestimate} establishes the theorem.
\end{proof}

\
\section{Lemmata}
In this section we collect lemmata that will be used in the proof of Theorem~\ref{thm:shiftedmoment}.

\begin{lemma}[Convexity bound]\label{lem:convexitybound}
For any $0 <\sigma<1$ and $\varepsilon >0$, there is a uniform bound
\begin{equation}\label{harcoslemma}
  L(\sigma\!+\!it,f)\ \ll_{\sigma,\varepsilon}\
  \mathfrak{q}(\tfrac{1}{2}\!+\!it,f)^{(1-\sigma)/2+\varepsilon},
\end{equation}
where $\mathfrak{q}(s,f)$ denotes the analytic conductor of $L(s,f)$, and the
implied constant depends only on $\sigma$ and $\varepsilon$.
\end{lemma}

\begin{proof}
See~\cite[Section 1.2]{harcos}. The uniform bound is deduced by considering upper bounds on $L(\sigma\!+\!it,f)$ in the half-plane $\sigma >1$ and $\sigma<0$ and then interpolating between the two via the Phragm\'{e}n-Lindel{\"o}f convexity principle.
\end{proof}

We note that while the implied constant in \eqref{harcoslemma} does not depend on the (fixed) form $f$,  this independence is not a necessity for the present work.

\begin{lemma}\label{lem:weightsum}
Let $\alpha,\beta\in\C$ with $|\alpha|,|\beta| \ll 1/\log{T}.$ Then, as $T\to \infty$, we have
\begin{equation}
\sum_{n\geq1}\frac{\left|a_f(n)\right|^2}{n^{1+\alpha+\beta}}e^{-2n/T} \ = \  L(1\!+\!\alpha\!+\!\beta,f\!\times\! \overline{f})
  +c_{f}\Gamma(-\alpha\!-\!\beta)\left(\frac{T}{2}\right)^{-\alpha-\beta}
  +O\left(T^{-1/2}\right),
\end{equation}
and, for $|a| \ll 1/\log T$, we have
\begin{equation}\label{expsumbound}
\sum_{n\geq1}{\frac{|a_f(n)|^2}{n^a}e^{-2n/T}} \ll T
\end{equation}
as $T\to \infty$.
\end{lemma}

\begin{proof} The bound in \eqref{expsumbound} follows by partial summation.
For the first sum, we begin by writing the main term of the asymptotic expansion using the definition of the convolution sum in property \ref{hyp:convolution}. We have
\begin{align}
  \sum_{n=1}^{\infty} \frac{\left|a_{f}(n)\right|^2}{n^{1+\alpha+\beta}}e^{-2n/T}
  &\ = \  \frac{1}{2 \pi i} \sum_{n= 1}^{\infty}\frac{ \left|a_{f}(n)\right|^2}{n^{1+\alpha+\beta}}
  \int_{(2)} \Gamma(w) \left(\frac{2n}{T}\right)^{-w}\d w \\
  &\ = \ \frac{1}{2 \pi i} \int_{(2)} \Gamma(w) \left(\frac{T}{2}\right)^{w}
  \sum_{n= 1}^{\infty} \frac{\left|a_{f}(n)\right|^2}{n^{1+w+\alpha+\beta}}\d w \\
  &\ = \ \frac{1}{2 \pi i} \int_{(2)} \Gamma(w) \left(\frac{T}{2}\right)^{w}
  L(1\!+\!w\!+\!\alpha\!+\!\beta, f \!\times\! \overline{f})\d w.
\end{align}

As described in property~\ref{hyp:convolution}, we let $c_{f}$ denote the residue of
$L(s, f \!\times\! \overline{f})$ at $s\!=\!1$. Note that
\begin{equation}
\res_{w=-\alpha-\beta}\left[\Gamma(w)
\left(\frac{T}{2}\right)^{w}L(1\!+\!w\!+\!\alpha\!+\!\beta, f\!\times\!\overline{f})\right]
\ = \ c_{f}\Gamma(-\alpha\!-\!\beta)\left(\frac{T}{2}\right)^{-\alpha-\beta},
\end{equation}
and
\begin{equation}
\res_{w=0}\left[\Gamma(w)\left(\frac{T}{2}\right)^{w}
L(1\!+\!w\!+\!\alpha\!+\!\beta,f\!\times\!\overline{f})\right]
\ = \ L(1\!+\!\alpha\!+\!\beta,f\!\times\!\overline{f}).
\end{equation}
Hence, moving the contour of
integration to $\Re(w)=-1/2$ and noting that the contribution from the horizontal sides of the contour is zero by the exponential decay of the gamma factor and the finite order of
$L(s, f \times \overline{f})$ in vertical strips not containing $s=1$,  we have that
\begin{align}
\sum_{n=1}^{\infty} \frac{\left|a_{f}(n)\right|^2}{n^{1+\alpha+\beta}}e^{-2n/T}
  &\ = \
  \begin{aligned}[t]
    &c_{f}\Gamma(-\alpha\!-\!\beta)\left(\frac{T}{2}\right)^{-\alpha-\beta}
      +L(1\!+\!\alpha\!+\!\beta, f \!\times\! \overline{f}) \\
    &+ O\left(\int_{\left(-\frac{1}{2}\right)}\Gamma(w)\left(\frac{T}{2}\right)^{w}
      L(1\!+\!w\!+\!\alpha\!+\!\beta, f \!\times\! \overline{f})\d w \right).
  \end{aligned}
\end{align}
We now estimate the error term. Letting $w=u+iv$ as above, we have
\begin{align}
  \int_{\left(-\frac{1}{2}\right)}\Gamma(w)\left(\frac{T}{2}\right)^{w}
  L(w+1+\alpha+\beta, f \times \overline{f})\d w
  &\ll O\left(T^{-1/2}\right).
\end{align}
This completes the proof.
\end{proof}

\begin{lemma}\label{lem:truncatedsum}
  Let $\alpha,\beta\in\C$ with $|\alpha|,|\beta|\ll 1/{\log T}$. Then, as $T\to \infty$, we have
  \begin{align*}
    \sum_{n \leq T} \left|a_{f}(n) \right|^{2}n^{-1+\alpha+\beta}
    \ = \ L(1-\alpha-\beta, f \times \overline{f})
    +T^{\alpha+\beta}\frac{c_{f}}{\alpha+\beta}
    +O\left(T^{-1/4+o(1)\log T} \right).
  \end{align*}
We also have, for $|a|\ll 1/{\log T}$, that
\begin{equation*}
  \sum_{n \leq T} |a_{f}(n)|^{2}n^{\sigma+a}\ \ll\ 
  \begin{cases}
    T, & \sigma = 0, \\
    T^{3/2}, & \sigma = 1/2,
  \end{cases}
\end{equation*}
as $T\to \infty$.
\end{lemma}
\begin{proof}
  The second bound is immediate from routine summation by parts. We have
  \be\begin{aligned}[b]
    \sum_{n \leq T} |a_{f}(n)|^{2}n^{\sigma+a}
    &\ = \ T^{\sigma+a}\sum_{n \leq T}|a_{f}(n)|^{2}
    -(\sigma+a)\int^T_1u^{\sigma+a-1}\left(\sum_{n \leq u} |a_{f}(u)|^{2}\right)\d u \\
    &\ll T^{1+\sigma+a}-(\sigma+a)\int^{T}_{1}u^{\sigma+a}\d u
    \ll T^{1+\sigma+a},
  \end{aligned}\ee
  where here we have made use of ~\ref{hyp:coeffsquaredsum}. 

  As for the first sum, in view of ~\ref{hyp:convolution}, an application of
  Perron's formula (see \ref{lem:perron}) yields
\begin{align}\label{eq:perronapp}
  \sum_{n \leq T} \left|a_{f}(n) \right|^{2}n^{-1+\alpha+\beta}
  &\ = \ \begin{aligned}[t]
    &\frac{1}{2 \pi i} \int^{\kappa+i\gamma}_{\kappa-i\gamma}
    L(1-\alpha-\beta+w, f \times \overline{f})T^{w}\frac{\d w}{w} \\
    &+O\left(T^{\Re(\alpha+\beta)}\frac{\log(T)}{\gamma}
      +\frac{\tau(2T)^{2}}{T^{1-\Re(\alpha+\beta)}}
      \left(1+T\frac{\log \gamma}{\gamma}\right)\right),
  \end{aligned}
\end{align}
where $\kappa=\Re(\alpha+\beta)+1/\log(T)$.
Moving the line of integration to $-1/2$, we pick up
residues from the simple poles at $w=0$ and $w=\alpha+\beta$. We have
\be
\res_{w=0}L(1-\alpha-\beta+w, f \times \overline{f}) \frac{T^{w}}{w}
\ = \ L(1-\alpha-\beta, f \times \overline{f}),
\ee
and
\be
\res_{w=\alpha+\beta}L(1-\alpha-\beta+w, f \times \overline{f})\frac{T^w}{w}
\ = \ T^{\alpha+\beta}\frac{c_{f}}{\alpha+\beta}.
\ee
We now estimate the error from closing the contour of integration. This is
done by applying Lemma~\ref{lem:convexitybound} for
$L(s, f\times\overline{f})$ in the critical strip $0<\sigma<1$.
From this we have, fixing $\vep>0$,
\be
L(\sigma +it, f \times \overline{f})
\ll_{\vep, \sigma} t^{2(1-\sigma)+\vep}
\ee
for $0<\sigma<1$. Finally, we know that $L(s+it, f\times\overline{f}) \ll 1$
for $\Re(s)>1$ since the $L$-function is given by an absolutely
convergent Dirichlet series.
We proceed to estimate the error from the contour itself. We have
\begin{align}\label{eq:s2vertcont}
  &\begin{aligned}[b]
  \frac{1}{2 \pi i} \int^{-1/2+i\gamma}_{-1/2-i\gamma}
  L(1-\alpha-\beta+w, f \times \overline{f})T^{w}\frac{\d w}{w}
  &\ll \int^{\gamma}_{-\gamma}
  v^{1+2\Re(\alpha+\beta)+2\vep}T^{-1/2}\frac{\d v}{-\frac{1}{2}+|v|} \\
  &\ll\gamma^{1+2\Re(\alpha+\beta)+2\vep}T^{-1/2}\log\gamma,
  \end{aligned} \\\intertext{and}
  &\frac{1}{2 \pi i} \int^{\kappa\pm i\gamma}_{-1/2 \pm i\gamma}
  L(1-\alpha-\beta+w, f \times \overline{f})T^{w}\frac{\d w}{w} \notag\\
  &\hspace{.25in}\ll\ \int^{\pm i\gamma}_{-1/2 \pm i\gamma}
  L(1-\alpha-\beta+w, f \times \overline{f})T^{w}\frac{\d w}{w}
  +\int^{\kappa\pm i\gamma}_{\pm i\gamma}L(1-\alpha-\beta+w,f\times\overline{f})
  T^{w}\frac{\d w}{w} \notag\\
  &\hspace{.25in}\ll\ \gamma^{-1+2\Re(\alpha+\beta)+2\vep}
  +\kappa T^{\kappa}\gamma^{-1+2\Re(\alpha+\beta)+2\vep} \notag\\
  &\hspace{.25in}\ll\ \gamma^{-1+2\Re(\alpha+\beta)+2\vep}+\left(\Re(\alpha+\beta)
    +\frac{1}{\log T} \right)\gamma^{-1+2\Re(\alpha+\beta)+2\vep}
  T^{\Re(\alpha+\beta)+1/\log T}.
  \label{eq:s2horcont}
\end{align}
Put $\gamma=T^{1/4}$. Then we find that the right-hand side of \eqref{eq:s2vertcont} is
\begin{equation}
  \ll T^{-1/4+\Re(\alpha+\beta)/2+\vep/2}\log T
  \ll T^{-1/4+\vep/2\log T}
 \end{equation}
 and that the right-hand side of \eqref{eq:s2horcont} is
 \begin{equation}
   \ll T^{-1/4+\Re(\alpha+\beta)/2+\vep/2}
  +\left(\Re(\alpha+\beta)+\frac{1}{\log T} \right)
  T^{-1/4+3\Re(\alpha+\beta)/2+1/\log T+\vep/2}
  \ll T^{-1/4 +\vep/2}.
\end{equation}
Finally, we estimate the error term in~\eqref{eq:perronapp}. We have
\be\begin{aligned}[b]
  T^{\Re(\alpha+\beta)}\frac{\log(T)}{\gamma}
  +\frac{\tau(2T)^{2}}{T^{1-\Re(\alpha+\beta)}}
  \left(1+T\frac{\log \gamma}{\gamma}\right)
  &\ll T^{-1/4 +\Re(\alpha+\beta)}\log T+\frac{(2T)^{o(1)}}{T}
  \left(1+T^{1-1/4}\log T \right) \\
  &\ll\log T\left(T^{-1/4 +\Re(\alpha+\beta)}+T^{-1/4 +o(1)}\right) \\
  &\ll T^{-1/4+o(1)} \log T.
\end{aligned}\ee
This completes the proof.
\end{proof}

\begin{lemma}\label{lem:errorsum}
  Let $a\in\C$ with $|a|\ll1/\log T$ and $X\sim T$. Then
  \be
    \sum_{n\leq X}\left|a_f(n)\right|^2n^{-3/4+a}e^{-n/X}
    \ll X^{1/4}.
  \ee
\end{lemma}
\begin{proof}
We use property~\ref{hyp:coeffsquaredsum} to write
\be\begin{aligned}[b]
  &\sum_{n\leq X}\left|a_f(n)\right|^2n^{-3/4+a}e^{-n/X} \\
  &\ll\left(\sum_{n\leq X}\left|a_f(n)\right|^2\right)X^{-3/4}e^{-1}
  +\int_1^X\left(\sum_{n\leq u}\left|a_f(u)\right|^2\right)
  u^{-7/4}e^{-u/X}\d u
  +\frac1 X\int_1^X\left(\sum_{n\leq u}\left|a_f(u)\right|^2\right)u^{-3/4}
  e^{-u/X}\d u \\
  &\ll X^{1/4}+\int_1^Xu^{-3/4}e^{-u/X}\d u
  +\frac1 X\int_1^Xu^{1/4}e^{-u/X}\d u.
\end{aligned}\ee
Changing variables twice yields
\be\int_1^Xu^{-3/4}e^{-u/X}\d u \ = \
X^{1/4}\int_{1/X}^1 u^{-3/4}e^{-u}\d u
= X^{1/4}\int_1^X u^{-5/4}e^{-1/u}\d u\ll X^{1/4},\ee
and, similarly, $\int_1^X u^{1/4}e^{-u/X}\d u\ll X^{5/4}$.
\end{proof}

\begin{lemma}[Stirling estimate]\label{lem:stirling}
  Let $L(s,f)$ be of degree $m$, $L_1,L_2\in\{L,\overline L\}$, choose
  $\alpha,\beta\in\C$, and take notation as in~(\ref{eq:funeqasymmetric}).
  As $t\to\infty$,
  \be
  \Phi_{L_1}(\tfrac{1}{2}\!+\!\alpha\!+\!it)
  \Phi_{L_2}(\tfrac{1}{2}\!+\!\beta\!-\!it)
  \ = \
  Q^{-2(\alpha+\beta)}\left(\frac{t}2\right)^{-m(\alpha+\beta)}
  \left(1+O\left(\frac1{t}\right)\right)
  \ee
  and, for some fixed complex number $a$, there exists a complex number $A$
  independent of $t$ (though dependent on $a$) such that, as $t\ra\infty$,
  \be \Phi_L(a-it)
  \ = \ AQ^{1-2(a-it)}e^{-i m \,t}\left(\frac t2\right)^{m(1/2-a+it)}
  \left(1+O\left(t^{-1}\right)\right) \ll t^{m(1/2-a)}. \ee
\end{lemma}
\begin{proof}We have
  \be\begin{aligned}[b]
    \Phi_L(s)\ = \ \frac{\overline{L_\infty}(1-s,f)}{L_\infty(s,f)}
    &\ = \ Q^{1-2s}\prod_{j=1}^d\frac{\Gamma\left(\tfrac12(1-s)+\overline{\mu_j}\right)}
    {\Gamma\left(\tfrac12s+\mu_j\right)} \\
    &\ = \ Q^{1-2s}\pi^{-d}\prod_{j=1}^d
    \Gamma\left(\tfrac12(1-s)+\overline{\mu_j}\right)
    \Gamma\left(1-\frac s 2-\mu_j\right)
    \sin\left(\pi\left(\frac s 2 +\mu_j\right)\right).
  \end{aligned}\ee
  Let $a,b\in \C$ and $L_1,L_2=L$ (the other case is
  established in the same way). Then
  \be\begin{aligned}[b]
    &\Phi_L(a+it)\Phi_L(b-it)\ = \ \frac{\Phi_L(b-it)}{\Phi_{\overline L}(1-a-it)} \\
    &\ = \ Q^{2(1-a-b)}
    \prod_{j=1}^m
    \frac{
      \Gamma\left(\frac12-\frac b2+\frac{it}2+\overline{\mu_j}\right)
      \Gamma\left(1-\frac b2+\frac{it}2-\mu_j\right)
    }{
      \Gamma\left(\frac a 2+\frac{it}2 + \mu_j\right)
      \Gamma\left(\frac12+\frac a 2+\frac{it}2 - \overline{\mu_j}\right)}
    \frac{
      e^{i\pi\left(\frac {b-it} 2 +\mu_j\right)}-
      e^{-i\pi\left(\frac {b-it} 2 +\mu_j\right)}
    }{
      e^{i\pi\left(\frac {1-a-it} 2 + \overline{\mu_j}\right)}-
      e^{-i\pi\left(\frac {1-a-it} 2 + \overline{\mu_j}\right)}
    } \\
  &\ = \ Q^{2(1-a-b)}
  \prod_{j=1}^m
  \left(\frac t2\right)^{1-a-b+2(-\mu_j+\overline{\mu_j})}
  \left(1+O\left(t^{-1}\right)\right).
  \end{aligned}\ee
  Recalling that $\sum_j\Im(\mu_j)=0$ (by assumption) completes the first part
  of the proof.

  The bound $\Phi_L(a-it) \ll t^{m(1/2-a)}$ follows in exactly the same way. We have
  \be\begin{aligned}[b]
    \Phi_L(a-it) &\ = \
    Q^{1-2(a-it)}\pi^{-m}\prod_{j=1}^m
    \Gamma\left(\frac12(1-a+it)+\overline{\mu_j}\right)
    \Gamma\left(1-\frac {a-it} 2-\mu_j\right)
    \sin\left(\pi\left(\frac {a-it} 2 +\mu_j\right)\right) \\
    &\ = \ Q^{1-2(a-it)}\pi^{-m}\prod_{j=1}^m
    \left(\frac{it}2\right)^{1/2-a+it-2\Im(\mu_j)}
    \exp\left(-\frac32+a-it+2\Im(\mu_j)\right)
    A_1e^{\pi t/2}\left(1+O\left(t^{-1}\right)\right) \\
    &\ = \ AQ^{1-2(a-it)}e^{-im\,t}\left(\frac t2\right)^{m(1/2-a+it)}
    \left(1+O\left(t^{-1}\right)\right),
  \end{aligned}\ee
as desired.
\end{proof}

\begin{lemma}\label{lem:GeneralMV}
  Let $\{a_n\}$ and $\{b_n\}$ be two sequences of complex numbers. For any real
  numbers $T$ and $H$, we have
  \be \int_T^{T+H}\left|\sum_{n\geq1} a_n n^{it}\right|^2\d t
  \ = \  H\sum_{n\geq1}\left|a_n\right|^2
  +O\left(\sum_{n\geq1}n\left|a_n\right|^2\right) \ee
  and
  \be \int_T^{T+H}\left(\sum_{n\geq1} a_n n^{-it}\right)
  \overline{\left(\sum_{n\geq1}b_nn^{-it}\right)}\d t
  \ = \  H\sum_{n\geq1}a_n\overline b_n
  +O\left(\left(\sum_{n\geq1} n\left|a_n\right|^2\right)^{\frac12}
    \left(\sum_{n\geq1}n\left|b_n\right|^2\right)^{\frac12}\right). \ee
\end{lemma}
\begin{proof}
This is a generalized form of Montgomery and Vaughan's large sieve proved in \cite[Lemma 1]{tsang}.
\end{proof}
\begin{lemma}\label{lem:parts}
  Let $\{a_n\}$ and $\{b_n\}$ be sequences of complex numbers. Let $T_1$ and
  $T_2$ be positive real numbers and $g(t)$ be a real-valued function that is
  continuously differentiable on the interval $\left[T_1,T_2\right]$. Then
  \begin{multline}
    \int_{T_1}^{T_2}g(t)\left(\sum_{n\geq1} a_n n^{-it}\right)
    \left(\sum_{n\geq1} b_n n^{it}\right)\d t
    \ = \  \\ \ = \
    \left(\int_{T_1}^{T_2} g(t)\d t\right)
    \sum_{n\geq1} a_n b_n
    +O\left(\left[\left|g\left(T_2\right)\right|
        +\int_{T_1}^{T_2}\left|g'(t)\right|\d t\right]
      \left[\sum_{n\geq1} n\left|a_n\right|^2\right]^{\frac12}
      \left[\sum_{n\geq1} n\left|b_n\right|^2\right]^{\frac12}\right).
  \end{multline}
\end{lemma}
\begin{proof} This is proved in \cite[Lemma 4.1]{mililemma}. \end{proof}


\section{Proof of Theorem \ref{thm:shiftedmoment}}\label{sec:shiftproof}

In this section we prove Theorem~\ref{thm:shiftedmoment}. We first construct an approximate
functional equation for $L(s\!+\!\alpha, f)L(1\!-\!s\!+\!\beta, \overline{f})$ using a method of Ramachandra~\cite[Theorem 2]{ramachandra}.

\begin{lemma}\label{lem:ram}
Let $\alpha \in \C, s\!=\!1/2+it,$ and $T\geq2$. Then, for $X=T/(2\pi)$ and $T\leq t \leq 2T$, we have
\begin{multline}\label{eq:ramaformula}
  L(s\!+\!\alpha,f) \ = \  \sum_{n=1}^{\infty}\frac{a_f(n)}{n^{s+\alpha}}e^{-n/X} +
  \ep_f\Phi_f(s\!+\!\alpha)\sum_{n\leq X}\frac{a_{\bar f}(n)}{n^{1-s-\alpha}} \\
  - \frac1{2\pi i}
  \int_{-3/4-i\infty}^{-3/4+i\infty}\ep_f\Phi_f(s\!+\!\alpha\!+\!w)\left(\sum_{n>X}
    \frac{a_{\bar f}(n)}{n^{1-w-s-\alpha}}\right)\Gamma(w)X^w\d w \\
  - \frac1{2\pi i}
  \int_{1/2-i\infty}^{1/2+i\infty}\ep_f\Phi_f(s\!+\!\alpha\!+\!w)\left(\sum_{n\leq X}
    \frac{a_{\bar f}(n)}{n^{1-w-s-\alpha}}\right)\Gamma(w)X^w\d w.
\end{multline}
\end{lemma}

\begin{proof}
We first use the Mellin identity
\begin{equation}\label{eq:mellin}
e^{-t}\ = \ -\frac{1}{2\pi}\int_{2-i\infty}^{2+i\infty}\Gamma(w)t^{-w}\d w,
\end{equation}
which holds for $t\geq0$, to write
\begin{equation}\label{eq:first}
\sum_{n=1}^\infty \frac{a_f(n)}{n^{s+\alpha}}e^{-n/X}
\ = \ \frac{1}{2\pi i}\int_{2-i\infty}^{2+i\infty} L(s\!+\!\alpha\!+\!w,f)\Gamma(w)X^w\d w.
\end{equation}
Here the interchanging of the summation and integral is justified by the absolute convergence of the Dirichlet series. On the other hand, by shifting the line of integration from $\Re(w)=2$ to $\Re(w)=-3/4$, we find
\begin{equation}\label{eq:second}
\sum_{n=1}^\infty \frac{a_f(n)}{n^{s+\alpha}}e^{-n/X}\ = \
L\left(\!s\!+\!\alpha,f\right)+\frac{1}{2 \pi i}\int_{-3/4-i\infty}^{-3/4+i\infty}{L(s\!+\!\alpha\!+\!w,f)\Gamma(w)T^w\d w},
\end{equation}
where $L\left(s\!+\!\alpha,f\right)$ is the residue of the simple pole of the integrand of the right-hand side of \eqref{eq:first} at $w\!=\!0$. By the functional equation
for $L(s,f)$ and the absolute convergence of the Dirichlet series of $L(1\!-\!s,\bar f)$, we have
\begin{align}\label{eq:third}
L(s\!+\!\alpha\!+\!w,f) &\ = \  \ep_f\Phi_f(s\!+\!\alpha\!+\!w)L(1\!-\!s\!-\!\alpha\!-\!w,\bar{f})\notag\\
    & \ = \  \ep_f\Phi_f(s\!+\!\alpha\!+\!w)\sum_{n\leq X}
    \frac{a_{\bar f}(n)}{n^{1-w-s-\alpha}} + \ep_f\Phi_f(s\!+\!\alpha\!+\!w)\sum_{n>X}
    \frac{a_{\bar f}(n)}{n^{1-w-s-\alpha}}.
\end{align}
Replacing $L(s\!+\!\alpha\!+\!w,f)$ in \eqref{eq:second} with the right-hand side of \eqref{eq:third}, we have
\begin{multline}\label{eq:fourth}
\sum_{n=1}^\infty \frac{a_f(n)}{n^{s+\alpha}}e^{-n/X}\ = \
L\left(\!s\!+\!\alpha,f\right)+\frac{1}{2 \pi i}\int_{-3/4-i\infty}^{-3/4+i\infty}{\ep_f\Phi_f(s\!+\!\alpha\!+\!w)\sum_{n\leq X}
    \frac{a_{\bar f}(n)}{n^{1-w-s-\alpha}}\Gamma(w)T^w\d w}\\
    +\frac{1}{2 \pi i}\int_{-3/4-i\infty}^{-3/4+i\infty}{\ep_f\Phi_f(s\!+\!\alpha\!+\!w)\sum_{n>X}
    \frac{a_{\bar f}(n)}{n^{1-w-s-\alpha}}\Gamma(w)T^w\d w}.
\end{multline}
Finally, we shift the line of integration of the integral involving the Dirichlet series over $n\leq X$ from $\Re(w)=-3/4$ to $\Re(w)=1/2$. We once again pass over the pole at $w=0$ and recover the residue
\begin{equation}
\ep_f\Phi_f(s\!+\!\alpha)\sum_{n\leq X}\frac{a_{\bar f}(n)}{n^{1-s-\alpha}}.
\end{equation}
Upon rearranging terms, we deduce the claimed approximate functional equation for $L(s\!+\!\alpha,f)$.
\end{proof}

Applying Lemma~\ref{lem:ram} to $L(1\!-\!s\!+\!\beta, \overline{f})$, where $\beta\in \C$, we find
\be\label{eqn:mixedfunctional}
\begin{aligned}[b]
  L(s\!+\!\alpha,f)L(1\!-\!s\!+\!\beta, \overline{f})\ := \ &
  \sum_{n=1}^{\infty}\frac{a_f(n)}{n^{s+\alpha}}e^{-n/X}
  \sum_{n= 1}^{\infty}\frac{a_{\bar{f}}(n)}{n^{1-s+\beta}}e^{-n/X} \\
  &+\Phi_f(s\!+\!\alpha)\Phi_{\bar{f}}(1\!-\!s\!+\!\beta)\sum_{n \leq X} \frac{a_{f}(n)}{n^{s-\beta}}\sum_{n\leq X}\frac{a_{\bar f}(n)}{n^{1-s-\alpha}}
  + \sum^{14}_{i=1} J_{i} \\
  \ := \ & S_1 + S_2 + \sum^{14}_{i=1} J_{i},
\end{aligned}\ee
say, where the $J_{i}$ for $1\leq i \leq 14$ denote the terms that arise as products of the integral
components of our mixed functional equation. In the next section, we show that $S_1$ and $S_2$ contribute to the main term in Theorem~\ref{thm:shiftedmoment}. The remaining $J_i$ terms are absorbed into the
error term. In Section ~\ref{section:error}, we give full details for the estimation of one of the $J_{i}$; the other estimations follow similarly or by an application of Cauchy-Schwartz.


\subsection{Main Term Calculations} \

Recalling $s=1/2+it$, let
\be
S_{1}\ := \  \left(\sum_{n\geq1}a_f(n)n^{-s-\alpha}e^{-n/X}\right)
\left(\sum_{n \geq 1}a_{\overline{f}}(n)n^{-1+s-\beta}e^{-n/X}\right).
\ee
Directly applying Lemma~\ref{lem:GeneralMV} yields
\be
\int^{2T}_{T} S_{1}\d t\ = \ T \sum_{n} \left|a_{f}(n)\right|^2
n^{-1-\alpha-\beta}e^{-2n/X}+
O\left(
    \left(\sum_{n} \left|a_{f}(n)\right|^2e^{-2n/X}n^{-2\Re(\alpha)}\right)^{\frac12}
    \left(\sum_{n} \left|a_f(n)\right|^2e^{-2n/X}n^{-2\Re(\beta)}\right)^{\frac12}
  \right).
\ee
With $X \sim T$, Lemma~\ref{lem:weightsum} allows us to conclude that
\be
  \int^{2T}_{T} S_{1}\d t
  \ = \ T\left(c_{f} \Gamma(-\alpha-\beta)\left(\frac{X}{2}\right)^{-\alpha-\beta}
    +L(1+\alpha+\beta, f \times \overline{f})\right)
  +O(T).
\ee
Supposing $|\alpha|, |\beta| \ll 1/\log T$,
\be
\int_T^{2T}S_1\d t\ \ll\ T \log T.
\ee

We now turn to the second product contributing to the main term in our shifted
moment result. Let
\be S_2 \ := \
\left( \ep_f\Phi_f(s+\alpha)
  \sum_{n\leq X}a_{\overline f}(n)n^{s+\alpha-1}\right)
\left(\ep_{\overline{f}}\Phi_{\overline{f}}(1-s+\beta)
  \sum_{n \leq X} a_{f}(n)n^{-s+\beta} \right).
\ee
Recalling that $\left|\ep_f\right|=1$, we use
Lemma~\ref{lem:stirling} to break up the integral and write
\be\begin{aligned}[b]
\int^{2T}_{T} S_{2}\d t
&\ = \ \int^{2T}_{T} \Phi_{f}
\left(\frac{1}{2}+\alpha + it \right)\Phi_{\overline{f}}
\left(\frac{1}{2}+\beta-it\right)
\left(\sum_{n \leq X} a_{\overline f}(n)n^{s+\alpha-1}\right)
\left(\sum_{n \leq X} a_{f}(n)n^{-s+\beta}\right)\d t \\
&\ = \ \begin{multlined}[t][0.8\textwidth]
  \int_T^{2T}\left(\frac{t}{2\pi}\right)^{-2(\alpha+\beta)}
  \left(\sum_{n \leq X} a_{\overline f}(n)n^{s+\alpha-1}\right)
  \left(\sum_{n \leq X} a_{f}(n)n^{-s+\beta}\right)\d t \\
  + O\left(\int^{2T}_{T} \left(\frac{t}{2\pi}\right)^{-2(\alpha+\beta)-1}
    \left(\sum_{n \leq X} a_{\overline f}(n)n^{s+\alpha-1}\right)
    \left(\sum_{n \leq X} a_{f}(n)n^{-s+\beta}\right)\right)\d t.
\end{multlined}
\label{eq:s2funerror}
\end{aligned}\ee
We evaluate the main term using \cite[Lemma 4.1]{mililemma}. We have
\be\begin{aligned}[b]\label{eq:s2mainterm}
  &\int_T^{2T}\left(\frac{t}{2\pi}\right)^{-2(\alpha+\beta)}
  \left(\sum_{n \leq X} a_{\overline f}(n)n^{s+\alpha-1}\right)
  \left(\sum_{n \leq X} a_{f}(n)n^{-s+\beta}\right)\d t \\
  &\ = \  \left(\int^{2T}_{T} \left(\frac{t}{2\pi}\right)^{-2(\alpha+\beta)}\d t \right)
  \sum_{n\leq X} \left|a_{f}(n) \right|^{2}n^{-1+\alpha+\beta} + \\
  &+O
  \left(
    \left(\left(\frac{T}{\pi}\right)^{-2(\alpha+\beta)}
      +\int^{2T}_{T}2|\alpha+\beta|
      \left(\frac{t}{2\pi}\right)^{-2\Re(\alpha+\beta)-1}\d t\right)
    \left(\sum_{n\leq X}\left|a_{f}(n)\right|^2n^{2\Re(\alpha)}\right)^{\frac12}
    \left(\sum_{n \leq X}\left|a_{f}(n)\right|^2n^{2\Re(\beta)}\right)^{\frac12}
  \right).
\end{aligned}\ee
Since $|\alpha|, |\beta| \ll 1/\log T$, we have
\be\left(\frac{T}{\pi}\right)^{-2(\alpha+\beta)}\ \ll\ 1\ee
and
\be
\int^{2T}_{T}2|\alpha+\beta|\left(\frac{t}{2\pi}\right)^{-2(\alpha+\beta)-1}\d t\ \ll\ 1.
\ee
Lemma~\ref{lem:truncatedsum} contains the estimates that allow us to
write~\eqref{eq:s2mainterm} as
\begin{multline}
\int^{2T}_{T} \left(\frac{t}{2\pi}\right)^{-2(\alpha+\beta)}
\left(\sum_{n \leq X} a_{\overline f}(n)n^{s+\alpha-1}\right)
\left(\sum_{n \leq X} a_{f}(n)n^{-s+\beta}\right)\d t \\
\ = \ \left(\int^{2T}_{T} \left(\frac{t}{2\pi}\right)^{-2(\alpha+\beta)}\d t \right)
\left(L(1-\alpha-\beta, f \times \overline{f})
  + X^{\alpha+\beta}\frac{c_{f}}{\alpha+\beta} \right)
+O(T) \ll T \log T.
\end{multline}


\subsection{Error Term Estimates}\label{section:error}

In this section we estimate a representative error term.
The remaining $J_{i}$ follow from a direct application of
Cauchy-Schwartz or from a similar argument.
Observe that none of the $J_i$ error terms contribute to the main term,
and none dominate the term 
\be \mathcal{J} \ := \  \left(\sum_{n\geq1}a_{\overline f}(n)n^{-s-\alpha}e^{-n/X}\right)
\left(\frac{1}{2\pi i}
  \int_{\left(\frac14\right)} \ep_{\overline{f}}
  \Phi_{\overline{f}}(1-s+\beta+w)
  \left(\sum_{n \leq X} a_{f}(n)n^{w-s+\beta} \right)\Gamma(w)X^{w}\d w
\right). \ee
This term does not occur in~\eqref{eq:ramaformula}, but it is clear that any $J_{i}$ is of the same order or dominated by $\mathcal{J}$.
The real difficulty in estimating $\int_T^{2T} \mathcal{J} \d t$ follows from
\be\label{eq:compositeintest}
\int_T^{2T} \int_{\left(\frac14\right)}\Phiidual(1-s+\beta+w)\Gamma(w)X^w\d w\d t
\ll X^{1/4}T^{1/2}, \ee
which is immediate from the estimate in Section~\ref{sec:control}. Refer to Section~\ref{sec:control} for full details.

Returning to $\int_T^{2T} \mathcal{J} \d t$, we interchange the sum and the integral.
Lemma~\ref{lem:parts} applies only to real-valued functions $g(t)$, so we
apply an absolute value, apply Lemma~\ref{lem:parts}, and then use
Lemma~\ref{lem:errorsum} and the estimate~\eqref{eq:compositeintest} to write
\be\begin{aligned}[b]
  &\int_T^{2T} \mathcal{J}\d t \\
  &\ = \ \int_T^{2T}
  \left(\sum_{n\geq1}\left|a_{\overline f}(n)\right|n^{-1/2-\alpha}e^{-n/X}\right)
  \left(\sum_{n \leq X} \left|a_{f}(n)\right|n^{-1/4+\beta} \right)
  \left(\frac{1}{2\pi i}
    \int_{\left(\frac14\right)} \ep_{\overline{f}}
    \Phi_{\overline{f}}(1-s+\beta+w)
    \Gamma(w)X^{w}\d w
  \right)\d t \\
  &\ll X^{1/4}
  \int_T^{2T}\frac{1}{2\pi i}\int_{\left(\frac14\right)}
    \left|\ep_{\overline{f}}\Phi_{\overline{f}}(1-s+\beta+w)\Gamma(w)X^{w}\right|\d w\d t
  +O(\mathfrak{h})
  = O\left(X^{1/2}T^{1/2}+\mathfrak{h}\right),
\end{aligned}\ee
where, if we put
\be g(t) \ := \  \int_{\left(\frac14\right)}
\left|\ep_{\overline{f}}\Phi_{\overline{f}}(1-s+\beta+w)\Gamma(w)X^{w}\right|\d w \ee
then
\be\mathfrak{h}
\ = \ \left(|g(2T)|+\int_T^{2T}\left|g'(t)\right|\d t\right)
\left(\sum_{n\geq1} \left|a_{\overline f}(n)\right|^2n^{-2\Re(\alpha)}e^{-2n/X}\right)
\left(\sum_{n\leq X}\left|a_{f}(n)\right|^2n^{\frac12+2\Re(\beta)}\right).\ee
The estimates for the sums are contained in Lemmas~\ref{lem:weightsum}
and~\ref{lem:truncatedsum}.
Utilizing the estimate in Lemma~\ref{lem:stirling} to
differentiate the computations in Section~\ref{sec:control} under the integral, we obtain
\be \int_T^{2T}\left|g'(t)\right|\d t\ \ll\ X^{\frac14}T^{-\frac12}.\ee
Thus, letting $X\sim T$,
\be \mathfrak{h}\ = \ X^{3/2}T^{-1/2} \sim T \ee
and we have proved that
\be \int_T^{2T}\mathcal{J}\d t\ \ll \ T. \ee


\appendix

\section{Additional lemmata and calculations}
In this appendix we state an effective version of Perron's Formula used in the proof of Lemma \ref{lem:truncatedsum}. We also provide the details for bounding $\int_{\left(\frac14\right)}\Phi_{\overline{f}}(1-s+\beta+w)\Gamma(w)\d w$, which is used in Section~\ref{section:error}.
\subsection{An Effective Perron Formula}

\begin{lemma}[Effective Perron]\label{lem:perron}
  Let $F(s)\ := \ \sum_{n=1}^\infty a_nn^{-s}$ be a Dirichlet series with finite abscissa
  of absolute convergence $\sigma_a$. Suppose that there exists some real number
  $\alpha \geq 0$ s.t.
  \begin{align*} (i)&\quad \sum_{n=1}^\infty |a_n| n^{-\sigma} \ll
    (\sigma - \sigma_a)^{-\alpha}&(\sigma>\sigma_a), \\
    \noalign{\text{and that $B$ is a non-decreasing function satisfying}}\\
    (ii)&\quad |a_n|\leq B(n)&(n\geq1). \end{align*}
  Then for $X\geq2, \gamma\geq 2,\sigma\leq\sigma_a,\kappa\ := \ \sigma_a-\sigma+1/\log X$,
  we have
  \be\label{eq:perron}
  \sum_{n\leq X}\frac{a_n}{n^s}\ = \ \frac{1}{2\pi i}\int_{\kappa-i\gamma}^{\kappa+i\gamma}
  F(s+w)X^w\frac{\d{w}}{w}
  +O\left(X^{\sigma_a-\sigma}\frac{(\log X)^\alpha}{\gamma}
    +\frac{B(2X)}{X^\sigma}\left(1+X\frac{\log \gamma}{\gamma}\right)\right).
  \ee
\end{lemma}
\begin{proof} \cite[\S II.2.1, Corollary 2.1]{tenenbaum}. \end{proof}

\subsection{Controlling
  $\int_{\left(\frac14\right)}\Phi_{\overline{f}}(1-s+\beta+w)\Gamma(w)\d w$}\label{sec:control}
Recall that $|\beta|\ll 1/ \log T$ and $s=1/2+it$.
In this section we provide the details for the bound
\be \int_{\left(\frac14\right)}\Phiidual(1-s+\beta+w)\Gamma(w)\d w\ \ll\
t^{-\frac12}. \ee
Since
\be\Phi_f(s)\ = \ \frac{L_\infty(1-s,\overline f)}{L_\infty(s,f)},\ee
we have the relation \be\Phi_f(s)\ = \ \frac{1}{\Phi_{\overline f}(1-s)}.\ee
To ease our application of Stirling's formula by ensuring that we are always
applying it as $t\ra+\infty$, we begin by breaking up the integrals as
\begin{align}
  &\int_{\left(\frac14\right)}\Phiidual(1-s+\beta+w)\Gamma(w)\d w \notag\\
  &\ = \ \int_{-\infty}^t\Phiidual\left(\frac34+\beta+i(v-t)\right)\Gamma\left(\frac14+iv\right)\d v
  +\int_t^\infty\Phiidual\left(\frac34+\beta+i(v-t)\right)
  \Gamma\left(\frac14+iv\right)\d v.
\end{align}
We consider the first integral, which, changing variables, equals
\begin{align}
  &\int_{-t}^\infty\Phiidual\left(\frac34+\beta-i(v+t)\right)
  \Gamma\left(\frac14-iv\right)\d v \notag\\
  &\ = \ \int_{0}^\infty\Phiidual\left(\frac34+\beta-iv)\right)
  \Gamma\left(\frac14-i(v-t)\right)\d v \notag\\
  &\ = \ \int_{0}^t\Phiidual\left(\frac34+\beta-iv)\right)
  \Gamma\left(\frac14-i(v-t)\right)\d v
  +\int_{t}^\infty\Phiidual\left(\frac34+\beta-iv)\right)
  \Gamma\left(\frac14-i(v-t)\right)\d v. 
\end{align}
We first consider
\begin{align}
  &\int_{0}^t\Phiidual\left(\frac34+\beta-iv)\right)
  \Gamma\left(\frac14-i(v-t)\right)\d v \notag\\
  &\ll\ \int_{0}^tv^{-1/2-2\beta}\Gamma\left(1/4+i(t-v)\right)\d v \notag\\
  &\ll\ \int_{0}^tv^{-1/2-2\beta}
  \exp\left(\left(-\frac14+i(t-v)\right)\log\left(\frac14+i(t-v)\right)
    -\left(\frac14+i(t-v)\right)\right)\d v\\ &:= \mathcal{A}_1, \notag
    \end{align}
say. For the principal branch, we have
\begin{align}
  &\exp\left(\left(-\frac14+i(t-v)\right)\log\left(\frac14+i(t-v)\right)
    -\left(\frac14+i(t-v)\right)\right) \notag\\
  &\ll\ (it)^{-1/4+i(t-v)}\exp\left(\left(-\frac14+i(t-v)\right)
    \log\left(1-\frac i{4t}-\frac v t\right)-\frac14\right) \notag\\
  &\ll\ (it)^{-1/4+i(t-v)}\exp\left[\left(-\frac14+i(t-v)\right)
  \left(-\frac1 t\left(\frac i4+v\right)\right)-\frac14\right] \notag\\
  &\ll\ (it)^{-1/4+i(t-v)}\notag\\
  & \ll t^{-1/4}e^{-\pi(t-v)/2}.
\end{align}
Since $|\beta|\ll1/\log T$, we find
\begin{align}
  \mathcal{A}_1\ &\ll\ t^{-1/4}\int_{0}^tv^{-1/2-2\beta}e^{-\pi(t-v)/2}\d v
  \ \ll\ t^{-1/4}\int_{0}^tv^{-1/2}e^{-\pi(t-v)/2}\d v \notag\\
  &=\ 2t^{-1/4}\sqrt{\frac2\pi}F\left(\sqrt{\frac{\pi t}2}\right)\notag\\
  &\ll t^{-3/4},
\end{align}
where $F$ is Dawson's integral, which satisfies
    $F\left(\sqrt{\pi t/2}\right)\ll t^{-1/2}$. Next, we have
    \begin{align}
  &\int_{t}^\infty\Phiidual\left(\frac34+\beta-iv)\right)
  \Gamma\left(\frac14-i(v-t)\right)\d v \notag\\
  &\ll\ \int_{t}^\infty\frac{v^{-1/2-2\beta}\d v}
  {\Gamma\left(\frac34+i(v-t)\right)e^{\pi(v-t)}} \notag\\
  &\ll\ \int_{t}^\infty v^{-1/2-2\beta}(v-t)^{-1/4}e^{-\pi(v-t)/2}\d v \notag\\
  &= \ \int_0^\infty(v+t)^{-1/2-2\beta}v^{-1/4}e^{-\pi v/2}\d v \notag\\
  &\ll\ t^{-1/2}\int_0^\infty v^{-1/4}e^{-\pi v/2}\d v\notag\\
 & \ll t^{-1/2}.
\end{align}Finally, we have
\begin{align}
  &\int_t^\infty\Phiidual\left(\frac34+\beta+i(v-t)\right)
  \Gamma\left(\frac14+iv\right)\d v \notag\\
  &\ = \ \int_1^\infty\frac{\Gamma\left(1/4+i(v+t-1)\right)\d v}
  {\Phii\left(\frac14-\beta-i(v-1)\right)} \notag\\
  &\ll\ \int_1^\infty\frac{(i(v+t-1))^{-1/4+i(v+t-1)}\d v}
  {v^{1-2(1/4-\beta+i)}} \notag\\
  &\ll\ \int_1^\infty(v+t-1)^{-1/4}v^{-1/2-2\beta} e^{-\pi(v+t-1)/2}\d v \notag\\
  &\ll\ t^{-1/4}\int_1^\infty v^{-1/2} e^{-\pi(v+t)/2}\d v \notag\\
  &\ = \ t^{-1/4}e^{-\pi(1+t)/2}\left[\sqrt 2 e^{\pi/2}
    \erfc\left(\sqrt{\frac\pi2}\right)\right]\notag \\
  &\ \ll\ t^{-1/4}e^{-\pi t/2},
\end{align}
where $\erfc(z)$ is the complementary error function.

\vspace{.25in}

\noindent {\it Acknowledgements.} This work was supported in part by NSF Grants DMS1265673 and DMS1347804 and Williams College. We thank Hung Bui, Steve Gonek, Gergely Harcos, Winston Heap, Micah Milinovich, Jeremy Rouse, Frank Thorne, and Ben Weiss for a number of valuable conversations, and the referee for helpful comments on an earlier version.

\end{document}